\documentclass[preprint,10pt]{elsarticle}
\usepackage{lipsum}
\makeatletter
\def\ps@pprintTitle{
 \let\@oddhead\@empty
 \let\@evenhead\@empty
 \def\@oddfoot{}
 \let\@evenfoot\@oddfoot}
 \renewcommand{\MaketitleBox}{
  \resetTitleCounters
  \def\baselinestretch{1}
  \begin{center}
    \def\baselinestretch{1}
    \Large \@title \par
    \vskip 18pt
    \normalsize\elsauthors \par
    \vskip 10pt
    \footnotesize \itshape \elsaddress \par
  \end{center}
  \vskip 12pt
}

\usepackage{amssymb}
\usepackage{amsthm}
\usepackage{lineno}
\usepackage{amsfonts}
\usepackage{amsmath}
\usepackage{graphicx}
\usepackage{cases}
\usepackage{epstopdf}
\usepackage{caption}
\usepackage{algorithmic}
\usepackage{mathabx}
\biboptions{sort&compress}

\ifpdf
  \DeclareGraphicsExtensions{.eps,.pdf,.png,.jpg}
\else
  \DeclareGraphicsExtensions{.eps}
\fi

\usepackage{amssymb}
\usepackage{mathtools}

\usepackage[mathscr]{eucal}
\usepackage{color}
\usepackage{geometry}
\usepackage{tikz}
\usetikzlibrary{shapes,arrows}
\usetikzlibrary{decorations.pathreplacing,angles,quotes}

\usepackage{todonotes}

\usepackage{dsfont}
\usepackage{cleveref}
\usepackage{verbatim}
\usepackage{graphicx}
\usepackage{subfig}
\usepackage{tikz}
\usepackage{pgfplots}

\usepgfplotslibrary{groupplots}

\newcommand{\mb}{\mathbb}

\newcommand{\mc}{\mathcal}

\newcommand{\mfT}{{\mathcal{T} }}

\newcommand{\mT}{\mathrm{T}}
\newcommand{\mS}{\mathrm{S}}
\newcommand{\mfF}{{\mathfrak{F}}}

\newtheorem{theorem}{Theorem}[section]
\newtheorem{lemma}[theorem]{Lemma}
\newtheorem{proposition}[theorem]{Proposition}

\numberwithin{equation}{section}

\newtheorem{remark}{Remark}
\newtheorem{assumption}{A\!\!}[section]

\begin{document}

\begin{frontmatter}
\title{Infinite-Dimensional LQ Mean Field Games with Common Noise: Small and Arbitrary Finite Time Horizons
\footnote{The authors would like to acknowledge the helpful comments of Michèle Breton and Minyi Huang.}\footnote{Funding: Dena Firoozi would like to acknowledge the support of the Natural Sciences and Engineering Research Council of Canada (NSERC), grants RGPIN-2022-05337 and DGECR-2022-00468. Hanchao Liu would like to acknowledge the support of IVADO through the NSERC-CREATE Program on Machine Learning in Quantitative Finance and Business Analytics (Fin-ML CREATE).}
}
\author[1]{Hanchao Liu} \author[2]{Dena Firoozi}
\address[1]{Department of Decision Sciences, HEC Montréal, Montreal, QC, Canada\\
(email: hanchao.liu@hec.ca)}
\address[2]{Department of Statistical Sciences, University of Toronto, Toronto, ON, Canada\\ (email: dena.firoozi@utoronto.ca)}
\begin{abstract}
We develop the theory of linear-quadratic (LQ) mean field games (MFGs) in Hilbert spaces with common noise modeled by an infinite-dimensional Wiener process that affects the dynamics of all agents. This common noise is modeled as an infinite-dimensional Wiener process affecting the dynamics of all agents. In the presence of common noise, the mean-field consistency condition is characterized by a system of coupled forward-backward \textit{stochastic} evolution equations (FBSEEs) in Hilbert spaces, whereas, in its absence it reduces to a system of coupled forward-backward \textit{deterministic} evolution equations. We establish the existence and uniqueness of solutions to the coupled linear FBSEEs associated with the LQ MFG framework for small time horizons and prove the $\epsilon$-Nash property of the resulting equilibrium strategy. Furthermore, we establish the well-posedness of these coupled linear FBSEEs for arbitrary finite time horizons. Beyond the specific context of MFGs, our analysis also yields a broader contribution by providing, to the best of our knowledge, the first well-posedness result for a class of infinite-dimensional linear FBSEEs, for which only mild solutions exist, over arbitrary finite time horizons.
\end{abstract}
\end{frontmatter}

\section{Introduction} 
Mean field game (MFG) theory studies dynamic games involving a large number of indistinguishable agents, each of whom becomes asymptotically negligible as the population size grows. In such games, agents interact only weakly, through the empirical distribution of their states or control actions. As the number of agents tends to infinity, this empirical distribution converges to the so-called mean field distribution. The optimal behavior of individuals in these large populations, as well as the resulting equilibrium, can then be approximated by the solution of the corresponding infinite-population game (see, e.g., \cite{huang2006large, huang2007large, lasry2007mean, carmona2018probabilistic, bensoussan2013mean, cardaliaguet2019master}).

 MFGs have been applied across a wide range of domains, most notably in financial markets, where they provide a powerful framework for modeling diverse problems. In particular, applications include systemic risk (\cite{carmona2013mean,Bo2015SystemicRiskInterbanking,chang2022Systemic}), price impact and optimal execution (\cite{casgrain_meanfield_2020, FirooziISDG2017, Cardaliaguet2018Meanfieldgame, huang2019mean}), portfolio trading (\cite{lehalle2019mean,Fu-Horst-2021,Fu-Horst-2022}), and equilibrium pricing (\cite{Firoozi2022MAFI, gomes_mean-field_2018, fujii_mean_2021}).

A central assumption in classical MFGs is the presence of idiosyncratic noise, typically modeled as independent Brownian motions affecting individual agents' dynamics. However, in many real-world settings, agents are also influenced by common sources of randomness--external factors that affect all agents simultaneously, induce dependencies among their dynamics, and impact their collective behavior. For instance, in financial models, macroeconomic factors such as monetary policy announcements, aggregate demand shocks, or systemic market events may be modeled as such common sources of randomness. Mathematically, in an MFG model, this phenomenon is typically represented by a Brownian motion that affects the dynamics of all agents, referred to as common noise. The presence of  common noise introduces additional challenges in analyzing and solving MFGs. In particular, the mean-field consistency condition is characterized by a system of coupled forward-backward stochastic differential equations (FBSDEs), which reduces to a deterministic system in the absence of common noise.

MFGs with common noise have been discussed and studied in, for instance, \cite{carmona2013mean,carmona2016mean,lacker2016general,BENSOUSSAN2015CommonNoise,Lions-common-noise-2007}. In analogy with the theory of stochastic differential equations, \cite{carmona2016mean} and  \cite{lacker2016general} formulate notions of strong and weak solutions for MFGs and establish the existence of weak solutions under broad assumptions. In these works, the evolution of the state distribution in the limiting case, where the number of agents tends to infinity, is captured as a random measure flow adapted to the filtration generated by the common noise. This research direction has been pursued in several subsequent studies, such as \cite{lacker2015translation, kolokoltsov2019mean, lacker2023closed}. Linear-quadratic (LQ) MFGs with common noise are studied in \cite{graber2016linear}, where the solution is presented both in terms of a system of forward-backward SDEs and via a pair of Riccati equations. Specifically, the mean field is characterized by the conditional expectation of the state of a representative agent  with respect to the filtration generated by the common noise in the limiting case. Other works on LQ MFGs with common noise include, but are not limited to, \cite{ahuja2016wellposedness, li2023linear, tchuendom2018uniqueness, bensoussan2021linear,ren2024risk}. A related development in MFG theory is the inclusion of a major agents into the framework. Unlike minor agents, whose influence diminishes as the number of agents grows, the impact of a major agent remains substantial and does not vanish as the population size approaches infinity (\cite{huang2010large, firoozi2020convex, carmona2016probabilistic,nourian2013epsilon,firoozi2022LQG,huang2020linear,carmona2017alternative}). Similarly to common noise, the presence of a major agent makes the mean field stochastic, adapted to the filtration generated by the major agent's Brownian motion.

 A recent line of research in MFGs aims to extend the theory to scenarios in which agents’ dynamics are modeled in infinite-dimensional spaces, as first theoretically studied in \cite{liu2025hilbert} and \cite{federico2024linear} for LQ settings. Such models enable the study of MFGs in which agents' behavior is governed by non-Markovian dynamics--such as delayed systems--by lifting the state process to an infinite-dimensional space, as first discussed and heuristically explored in \cite{fouque2018mean} and \cite{carmona2018systemic} in the context of systemic risk models. This research direction has been subsequently developed in recent works, including \cite{federico2024mean} for a class of nonlinear MFGs, \cite{firoozi-Kratsios-Yang-2025} for learning solutions to infinite-dimensional LQ MFGs using neural operators, \cite{particle-systems-hilbert-2025} for the optimal control of interacting particles in Hilbert spaces, \cite{jaber2023equilibrium} for related static MFGs involving control path-dependent cost functionals, and \cite{LQ-master-eq-Hilbert-2025} for LQ master equations in Hilbert spaces. There are also studies on mean field systems and MFGs in which the analysis leads to infinite-dimensional equations, such as \cite{dunyak2024quadratic,2GraphonGames2025,hetero-MF-sys-2025,Hetero-MF-sys-Pham-2025}, that differ from MFGs defined in Hilbert spaces.
 
 To the best of our knowledge, infinite-dimensional MFGs with common noise have not yet been explored in the literature. In this work, we study such games within an LQ framework. More specifically, the main contributions of this paper are summarized as follows. 
\begin{itemize}
    \item[(i)] We study a general \( N \)-player LQ game with common noise in Hilbert spaces. Specifically, the state dynamics of each agent are governed by a linear stochastic evolution equation driven by both idiosyncratic and common noises, modeled as infinite-dimensional Wiener processes with distinct covariance operators. The drift and diffusion coefficients of the dynamics depend on both the individual state and the average state of the population. Each agent seeks to choose its strategy so as to minimize a quadratic cost functional. We establish regularity conditions for a system of $N$ coupled semilinear stochastic evolution equations in Hilbert spaces, driven by both idiosyncratic and common noises, to ensure the well-posedness of the $N$-player games under consideration. 
    \item[(ii)] We study the limiting problem as the number of agents $N$ tends to infinity and characterize a Nash equilibrium strategy for a representative agent. This involves establishing a unique fixed point to the mean-field consistency condition, which is characterized by a system of coupled forward-backward \textit{stochastic} evolution equations (FBSEEs) in Hilbert spaces. The presence of common noise makes this problem more mathematically involved, whereas in its absence one only needs to address the well-posedness of a system of coupled forward-backward \textit{deterministic} evolution equations, as studied in \cite{liu2025hilbert} and \cite{federico2024linear}. Furthermore, we establish the 
$\epsilon$-Nash property for the obtained equilibrium strategy.
    \item[(iii)] We establish the existence and uniqueness of solutions to the resulting system of coupled linear FBSEEs associated with the LQ MFG framework for both \textit{small} and \textit{arbitrary} finite time horizons. In contrast, in our previous work on Hilbert-space valued LQ MFGs without common noise \cite{liu2025hilbert}, this was established only for a \textit{small} time horizon. Beyond the specific context of MFGs, our analysis also yields a broader contribution by providing, to the best of our knowledge, the first well-posedness result for a class of coupled linear FBSEEs in Hilbert spaces 
over arbitrary finite time horizons. To establish existence, we identify a decoupling field and develop an analytical approach for coupled linear FBSEEs in Hilbert spaces, where only mild, rather than strong, solutions exist. This approach differs from that in \cite{federico2024linear}, which considers a class of coupled forward-backward \textit{deterministic} evolution equations. Moreover,  our framework applies to a broader class of Hilbert space-valued LQ MFGs with stochastic diffusion coefficients 
that also incorporate common noise. Naturally, it also applies to the MFG models studied in \cite{liu2025hilbert} and \cite{federico2024linear}. 
\end{itemize}

The organization of this paper is as follows. \Cref{sec:notation} introduces the notations used in the paper. \Cref{cabscom} presents regularity results for coupled semilinear stochastic evolution equations with common noise in Hilbert spaces. \Cref{HLMFGcom} studies LQ MFGs with common noise in Hilbert spaces, including the resulting stochastic control problem in the limiting case, the well-posedness of the mean-field consistency equations over small time horizons, and the Nash equilibrium strategy and its 
$\epsilon$-Nash property. Finally, \Cref{arbit-time} establishes the well-posedness of the mean-field consistency equations over arbitrary finite time horizons.

\section{List of Notations}\label{sec:notation}
\begin{itemize}
    \item $\mfT=[0,T]$: Time horizon;  $\mathcal{N}=\{1, 2, \ldots, N\}$: index set; $\mathbb{N}=\{1, 2, \ldots\}$: Set of positive integers.
    \item $H, U, V$: Separable Hilbert spaces. Each space is equipped with its own norm, defined as 
    $\left| x \right|_{H,U,V} = \sqrt{\langle x,x \rangle_{H,U,V}}$ and has an orthonormal basis $\{e_i^{H,U,V}\}_{i\in \mathbb{N}}$. For brevity, the space-specific indices are omitted when the context is clear.
    \item $H^{N}$: $N$-product space of $H$ equipped with the product norm $ \left|\textbf{x} \right|_{H^{N}}=\left (\sum_{i\in \mc{N}} \left|x_i \right|_{H}^{2}  \right  )^{\frac{1}{2}}$ for all $\textbf{x}=(x_1, \cdots, x_N)\in H^N$.
    \item $\mT^\ast$: Adjoint of operator $\mT$; $\Sigma(H)$: Set of self-adjoint operators; $\Sigma^{+}(H)$: Set of positive operators.
    \item $\mathcal{L}(V;H)$: Space of bounded linear operators $\mT$ from $V$ to $H$ equipped with the norm $\left \| \mT \right \|_{\mathcal{L}(V;H)}=\sup_{\left | x \right |_{V}=1}\left | \mT x \right |_{H}$; $\mathcal{L}(H):=\mathcal{L}(H;H)$.
     \item $\mathcal{L}_1(V;H)$: Space of trace class operators from $V$ to $H$; $\mathcal{L}_1(V):=\mathcal{L}_1(V;V)$.
    \item $V_Q\!=\!Q^{\frac{1}{2}}V$: Separable Hilbert space endowed with the inner product $
\left \langle \!u,v\! \right \rangle_{V_{Q}}\!\!\!=\!\!\sum_{j \in \mathbb{N}:\lambda_j>0}\frac{1}{\lambda_j}\left \langle\! u,e_j \!\right \rangle_V\!\left \langle \!v,e_j \!\right \rangle_V$, for every $u,v \in V_Q 
$ and the positive operator $Q\in \mathcal{L}_1(V)$, i.e. $Q\in \mathcal{L}_1(V)$ is self-adjoint and satisfies $\langle Qx,x\rangle \geq 0,\, \forall x \in V$.
    \item $\mathcal{L}_2(V;H)$: Space of Hilbert-Schmidt operators from $V$ to $H$ equipped with the inner product $\left \langle \mT,\mS  \right \rangle_2 :=\sum_{i \in\mathbb{N}}  \left \langle \mT e_i,\mS e_i  \right \rangle_{H}$ for all $\mT,\mS \in \mathcal{L}_2(V;H)$. 
   \item $\mathcal{B}(\mathcal{X})$: Borel $\sigma$-algebra on the Banach space $\mathcal{X}$.
 \item \!$\mathcal{M}^2_\mathcal{F}(\mfT;\mathcal{X})$: Banach 
space of all $\mathcal{X}$-valued progressively measurable processes $x=\{x(t): t \in \mfT\}$ with respect to the filtration $\mc{F}=\{\mc{F}_t: t \in \mfT\}$ satisfying 
$
    \left|x  \right|_{\mathcal{M}^2_\mc{F}(\mfT;\mathcal{X})}\!:=\left (\mathbb{E}\int_{0}^{T}\left|x(t) \right|^{2}_{\mathcal{X}}dt  \right )^{\frac{1}{2}}< \infty.$
    \item $\mathcal{H}^2_\mc{F}(\mfT;\mathcal{X})$: Banach 
space of all $\mathcal{X}$-valued progressively measurable processes $x=\{x(t): t \in \mfT\}$ with respect to the filtration $\mc{F}=\{\mc{F}_t: t \in \mfT\}$ satisfying 
$
\left|x \right|_{\mathcal{H}^2_\mc{F}(\mfT;\mathcal{X})}:= \Big (\displaystyle \sup_{t \in \mfT} \mathbb{E}\left|x(t) \right|_{\mathcal{X}} ^{2}\Big )^{\frac{1}{2}}< \infty.
$
\end{itemize}

\section{Coupled Semilinear Stochastic Evolution Equations with Common Noise in Hilbert Space} \label{cabscom}

In this section, we study the well-posedness of a set of $N$ coupled semilinear stochastic evolution equations driven by $N$ independent $Q$-Wiener processes \(\left\{W_{i} \right\}_{i \in \mc{N}}\) and a common $Q_0$-Wiener process \( W_0 \), all defined on $\big(\Omega, \mfF, \mathbb{P} \big)$, and taking values in the Hilbert space $V$. Here, $Q$ and $Q_0$ are positive covariance operators in $\mathcal{L}_1(V)$. 

\begin{remark} By applying the enumeration of $\mathbb{N}\times \mathbb{N}$ to the sequence of mutually independent real-valued Brownian motions $\left \{ \beta _j \right \}_{j \in\mathbb{N}}$, we can obtain infinitely many distinct sequences of Brownian motions $ \{ \beta^{i}_j \}_{j \in \mathbb{N}}=\{\beta _{1}^{i}, \beta _{2}^{i}, \dots, \beta _{j}^{i}, \dots \}$, each sequence indexed by $i\in\mathbb{N} \cup \left\{ 0\right\}$.
The infinite-dimensional common noise $W_0=\{W_0(t): t\in \mfT\}$, with $W_0(t) \in V$ and the covariance operator $Q_0$, may be constructed using the first subsequence $ \{ \beta^{1}_j \}_{j\in \mathbb{N}}$, as    
\begin{equation}
W_{0}(t)=\sum_{j\in \mathbb{N}}\sqrt{\lambda^{0} _{j}}\beta ^{1}_{j}(t)e^{0}_{j},\quad t \in \mfT,
\end{equation} 
where \( \{ \lambda_j^0, e_j^0 \}_{j \in \mathbb{N}} \) denotes the set of the eigenvalues and the corresponding eigenvectors of \( Q_0 \). Moreover, the element $i$ of the sequence of independent $Q$-Wiener processes $\left\{W_{i}\right\}_{i\in \mathbb{N}}$, with $W_{i}(t)\in V$, is constructed as 
\begin{equation} \label{wiei}
W_{i}(t)=\sum_{j \in\mathbb{N}}\sqrt{\lambda _{j}}\beta ^{i+1}_{j}(t)e_{j},\quad t \in \mfT,   
\end{equation} 
where $\{ \beta^{i+1}_j \}_{j\in \mathbb{N}}$ denotes the $(i+1)$-st subsequence of Brownian motions and \( \{ \lambda_j, e_j \}_{j \in \mathbb{N}} \) denotes the set of the eigenvalues and the corresponding eigenvectors of \( Q\) (for more details see \cite[Prop. 3.1]{liu2025hilbert}.).
\end{remark}

Now consider a system of $N$ coupled stochastic evolution equations driven by idiosyncratic $Q$-Wiener processes and  influenced by a common $Q_0$-Wiener process. For each equation $i$, $i \in \mathcal{N}$, the state at time $t, \, t \in \mfT$, is represented by $x_i(t) \in H$ and satisfies 
\begin{align} \label{mildcoupcom}
x_i(t)=&S(t)\xi _i+\int_{0}^{t}S(t-r)F_i(r,\textbf{x}(r),u_i(r))dr+\int_{0}^{t}S(t-r)B_i(r,\textbf{x}(r),u_i(r))dW_{i}(r)\notag \\ &+\int_{0}^{t}S(t-r)B_0(r,\textbf{x}(r),u_i(r))dW_{0}(r), \quad   
\end{align}
where the corresponding control input is represented by $u_i(t) \in U$. The vector process $\textbf{x}=\{\textbf{x}(t)=(x_1(t), x_2(t), \ldots, x_N(t)):\, t\in \mfT\}$ is an $H^{N}$-valued stochastic process. The $C_{0}$-semigroup $S(t) \in \mathcal{L}(H),\, t \in \mfT$, is generated by an unbounded linear operator $A$, with domain $\mathcal{D}(A)$, such that 
\begin{equation}
\left\|S(t) \right\|_{\mathcal{L}(H)} \leq  M_T:= M_A e^{\alpha T},\quad \forall t \in \mfT,
\end{equation}
with $M_A \geq 1$ and $\alpha \geq 0$ \cite{goldstein2017semigroups}. The mappings $F_i: \mfT\times H^{N} \times U \rightarrow H$ and $B_i: \mfT\times H^{N} \times U \rightarrow \mathcal{L}_2(V_Q,H), \forall i \in \mathcal{N}$, are defined for all $i\in \mathcal{N}$. Moreover, the mapping $B_0$ is defined as $B_0: \mfT\times H^{N} \times U \rightarrow \mathcal{L}_2(V_{Q_{0}},H)$. 

The filtration $\mathcal{F}^0 = \{\mathcal{F}^0_t : t \in \mfT\}$ denotes the natural filtration generated by $W_0$, augmented with the collection $\mathfrak{N}$ of $\mathbb{P}$\nobreakdash-null sets in $\mfF$. Similarly, the filtration $\mathcal{F}^{[N],0} = \{\mathcal{F}^{[N],0}_t : t \in \mfT\}$ denotes the natural filtration generated by $\{W_i\}_{i \in \mathcal{N}}$, $W_0$ and the initial conditions $\{\xi_i\}_{i \in \mathcal{N}}$, also augmented with $\mathfrak{N}$\,\footnote{Due to \Cref{ass-add}, the filtrations $\mathcal F^0$ and $\mathcal F^{[N],0}$ satisfy the usual conditions.}.

To establish the well-posedness of the set of equations given by \eqref{mildcoupcom}, we impose the following assumptions.
\begin{assumption} \label{as1com}   
 $u_i \in \mathcal{M}^2_{\mathcal{F}^{[N],0}}(\mfT;U)$.
\end{assumption}
\begin{assumption} \label{ass-add}   
 The initial conditions $\{\xi_i\}_{i\in\mathcal N}$ are independent of $W_0$ and $\{W_i\}_{i\in\mathcal N}$.
\end{assumption}
\begin{assumption}
\label{as2com}
    The mappings $F_i: \mfT\times H^{N} \times U \rightarrow H $ is $\mathcal{B}\left ( \mfT \right ) \otimes \mathcal{B}(H^{N})\otimes \mathcal{B}(U)/ \mathcal{B}(H) $-measurable. 
    \end{assumption}
\begin{assumption}\label{as4com}
   The mappings $B_i: \mfT\times H^{N} \times U \rightarrow \mathcal{L}_2(V_Q,H)$ are $\mathcal{B}\left ( \mfT \right )\otimes\mathcal{B}(H^{N}) \otimes \mathcal{B}(U)/\mathcal{B}(\mathcal{L}_2(V_Q,H))$-measurable and the mapping $B_0: \mfT\times H^{N} \times U \rightarrow \mathcal{L}_2(V_{Q_{0}},H)$ is $\mathcal{B}\left ( \mfT \right )\otimes\mathcal{B}(H^{N}) \otimes \mathcal{B}(U)/\mathcal{B}(\mathcal{L}_2(V_{Q_{0}},H))$-measurable.
 \end{assumption}  
 \begin{assumption} \label{as5com}
 There exists a constant $C$ such that, for every $t \in \mfT$, $u \in U$, and $\textbf{x},\textbf{y} \in H^N$, we have
    \begin{align*}
     &\left|F_i(t,\textbf{x},u)-F_i(t,\textbf{y},u) \right|_H+\left\| B_i(t,\textbf{x}, u)-B_i(t,\textbf{y}, u)\right\|_{\mathcal{L}_2}+\left\| B_0(t,\textbf{x}, u)-B_0(t,\textbf{y}, u)\right\|_{\mathcal{L}_2}\leq C \left|\textbf{x}-\textbf{y} \right |_{H^{N}}, \notag \allowdisplaybreaks\\
    &\left|F_i(t,\textbf{x}, u) \right|^2_H+\left\| B_i(t,\textbf{x}, u)\right\|_{\mathcal{L}_2}^2+\left\| B_0(t,\textbf{x}, u)\right\|_{\mathcal{L}_2}^2\leq C^{2}\left (1+\left|\textbf{x} \right|^2_{H^{N}}+ \left|u \right|^{2}_U  \right ).    \notag
    \end{align*}
    \end{assumption}
 
\begin{theorem}(Existence and Uniqueness of a Mild Solution) \label{enucom}
Under \Cref{as1com}-\Cref{as5com}, the set of coupled stochastic evolution equations given by \eqref{mildcoupcom} admits a unique mild solution in the space $\mathcal{H}^{2}_{\mathcal{F}^{[N],0}}(\mfT;H^N)$.     
\end{theorem}
\begin{proof}
     This result is an adaptation of the corresponding result obtained in the absence of common noise, as presented in \cite{liu2025hilbert}. Hence, the theorem can be proved by following the same approach as in the proof of \cite[Thm. 3.3]{liu2025hilbert}.
\end{proof}
\section{Hilbert Space-Valued LQ Mean Field Games with Common Noise}\label{HLMFGcom}
\subsection{$N$-Player Game}\label{cabcom}
In this section, we consider a differential game in Hilbert spaces defined on $ \big(\Omega, \mfF, \mathcal{F}^{[N],0}, \mathbb{P} \big)$. This game involves $N$ asymptotically negligible (minor) agents, whose dynamics are governed by a system of coupled stochastic evolution equations, each of which is given by a linear form of \eqref{mildcoupcom}. Specifically, the dynamics of a representative agent indexed by  $i$, $i\in \mathcal{N}$, are given by 
\begin{align} \label{statecom}
x_i(t) = &S(t)\xi_i + \int_0^t S(t-r)(B u_i(r)+F_{1}x^{(N)}(r)) \, dr + \int_0^t S(t-r) (D x_i(r) +F_{2}x^{(N)}(r)+ \sigma) \, dW_i(r) \notag \\&+ \int_0^t S(t-r) (D_0 x_i(r) + F_{0}x^{(N)}(r)+\sigma_0) \, dW_0(r),
\end{align}
where $x_
i(t)$ and $u_i(t)$ denote, respectively, the state and the control input of agent $i$ at time $t$. The term $x^{(N)}(t):=\frac{1}{N}(\sum_{i\in \mc{N}}x_{i}(t) )$ represents the average state of the 
$N$ agents. Moreover, the operators involved include $B \in \mathcal{L}(U;H)$, $F_1 \in \mathcal{L}(H)$, $D$ and $F_2 \in \mathcal{L}(H; \mathcal{L}_2(V_{Q};H))$, $\sigma$ and $\sigma_0 \in \mathcal{L}(V;H)$, \(D_0\) and \(F_0 \in \mathcal{L}(H; \mathcal{L}_2(V_{Q_0};H))\).\footnote{A slight generalization of the operator spaces is considered, as described in \cite[Sec. 5.2.]{liu2025hilbert}.}

\begin{assumption}\label{init-cond-iid}The initial conditions $\{\xi_i\}_{i \in \mc{N}}$ are i.i.d. with $\mb{E}[\xi_i]=\bar{\xi}$ and $\mb{E} \left |\xi_i \right |^2 < \infty$,  and independent from $W_0$ and $\left\{W_{i} \right\}_{i \in \mc{N}}$.
\end{assumption}

\begin{assumption}(Filtration \& Admissible Control for $N$-Player Game)\label{clcom}
The set of admissible control laws for agent $i$, denoted by $\mc{U}^{[N],0}$, is defined as the collection of control laws $u_i$ that belong to $ \mathcal{M}^2_{\mc{F}^{[N],0}}(\mfT;U)$.    
\end{assumption}
The existence and uniqueness of a mild solution to \eqref{statecom} is guaranteed by \Cref{enucom}. In addition, agent $i$ aims to find a strategy minimizing the cost functional
\begin{equation} \label{cosfinNcom}
J^{[N]}_{i}(u_{i},u_{-i})=\mathbb{E}\int_{0}^{T}\left(\left|M^{\frac{1}{2}}\left (x_{i}(t)-\widehat{F}_{1} x^{(N)}(t)  \right ) \right|^{2}+\left|u_{i}(t) \right|^{2}\right)dt+\mathbb{E}\left|G^{\frac{1}{2}}\left (x_{i}(T)-\widehat{F}_{2} x^{(N)}(T)  \right ) \right|^{2},  
\end{equation}
where $M\in \mathcal{L}(H)$ and $G\in \mathcal{L}(H)$ are positive operators, and $\widehat{F}_1, \widehat{F}_2 \in \mathcal{L}(H)$. 
\subsection{Limiting Game}\label{sec:limiting:common-noise}
In the limiting case, where the number of agents $N$ approaches infinity, the state dynamics and cost functional of a representative agent $i$, $i\in\mathcal N$, are given, respectively, by 
\begin{align} \label{dylimi-S0com}
x_i(t)=&S(t)\xi_i+\int_{0}^{t}S(t-r)(Bu_i(r)+F_1\bar{x}(r))dr+\int_0^t S(t-r) (D x_i(r) + F_2\bar{x}(r)+\sigma) \, dW_i(r)\notag \\ &+ \int_0^t S(t-r) (D_0 x_i(r) + F_{0}\bar{x}(r)+\sigma_0) \, dW_0(r),
\end{align}
and 
\begin{equation} \label{cosfinlimi-S0com}
{J}^\infty_i(u_i)=\mathbb{E}\int_{0}^{T}\left(\left|M^{\frac{1}{2}}\left (x_{i}(t)-\widehat{F}_{1}\bar{x}(t)  \right ) \right|^{2}+\left| u_{i}(t) \right|^{2}\right)dt+\mathbb{E}\left|G^{\frac{1}{2}}\left (x_{i}(T)-\widehat{F}_{2}\bar{x}(T) \right ) \right|^{2},
\end{equation} 
where $\bar{x}(t)$ represents the coupling term in the limit, termed the mean field. We impose the following assumption for the limiting problem.  
\begin{assumption}(Filtration \& Admissible Control for Limiting Game)\label{Info-set-Adm-Cntrl-CN}
  The filtration $\mathcal{F}^{i,0, \infty}$ of agent $i$ is the natural filtration generated by $W_i$, $W_0$, and the initial condition $\xi_i$, and it satisfies the usual conditions. Subsequently, the set of admissible control laws for agent $i$, denoted by $\mc{U}^i$, is defined as the collection of control laws $u_i$ that belong to $ \mathcal{M}^2_{\mc{F}^{i,0,\infty}}(\mfT;U)$.  
\end{assumption}

We proceed with the following steps to address the described limiting MFG and to obtain a Nash equilibrium strategy for it. 

We define the Banach space $C_{\mathcal{F}^0}(\mfT; L^2(\Omega; H))$ as
\begin{equation} 
C_{\mathcal{F}^0}(\mfT; L^2(\Omega; H)) := \left\{ g : \mfT \to L^2(\Omega; H) \,\middle|\, g \text{ is } \mathcal{F}^0-\text{adapted  and } t \mapsto g(t) \text{ is continuous in }L^2(\Omega; H) \right\},  \notag  \end{equation} 
equipped with the standard supremum norm $| g |_{C_{\mathcal{F}^0}(\mathcal{T}; L^2(\Omega; H))} 
:= 
\sup_{t \in \mathcal{T}} 
\big( 
\mathbb{E}| g(t) |_H^2 
\big)^{1/2}$. Moreover, every $g \in C_{\mathcal{F}^0}(\mfT; L^2(\Omega; H))$ has a progressively measurable modification.

First, we treat the interaction term as an input \(g \in C_{\mathcal{F}^0}(\mfT; L^2(\Omega; H))\), and solve the resulting optimal control problem for a representative agent described by the dynamics
\begin{align} \label{dylimicom}
 dx_i(t)=&(Ax_i(t)+Bu_i(t)+F_{1}g(t))dt+(D x_i(t) +F_{2}g(t)+ \sigma) dW_{i}(t)\notag\\&+(D_0 x_i(t) + F_{0}g(t)+\sigma_0) dW_{0}(t) ,\hspace{0.5cm}  x_i(0)=\xi_i,  
\end{align}
and the cost functional 
\begin{equation} \label{cosfinlimicom}
J(u_i)=\mathbb{E}\int_{0}^{T}\left(\left|M^{\frac{1}{2}}\left (x_{i}(t)-\widehat{F}_{1}g(t)  \right ) \right|^{2}+\left| u_{i}(t) \right|^{2}\right)dt+\mathbb{E}\left|G^{\frac{1}{2}}\left (x_{i}(T)-\widehat{F}_{2}g(T)  \right ) \right|^{2}.
\end{equation} 
The solution of the above optimal control problem yields the optimal pair $(x^\circ_i, u^\circ_i)$.

Then, we obtain the resulting mean field, denoted by $y_g \in  C_{\mathcal{F}^0}(\mfT; L^2(\Omega; H))$, satisfying   
\begin{equation} \label{fixpdcom1}
\displaystyle \sup_{t\in \mfT}\mathbb{E}\left|y_g(t)- \frac{1}{N}\sum_{i\in \mc{N}}x^{\circ}_{i}(t) \right|_{H}^{2} \rightarrow 0,
\end{equation} 
where $x_{i}^{\circ}$ represents the optimal state of agent $i$ corresponding to the control problem described by \eqref{dylimicom}-\eqref{cosfinlimicom}.

Finally, we study the consistency condition
\begin{equation}
y_g(t) = g(t), \quad \forall t \in \mfT,\quad  \mathbb{P}-a.s.
\end{equation}
whose fixed point characterizes the mean field $\bar{x}$.
\subsubsection{Optimal Control Problem of a Representative Agent} \label{sinconcom}
Due to the homogeneity of the agents, we drop the index $i$ in this section. Each agent faces a stochastic control problem in Hilbert spaces described by the state evolution equation 
\begin{align} \label{dylimisincom}
x(t) = &S(t)\xi + \int_0^t S(t-r)(B u(r)+F_1g(r)) \, dr + \int_0^t S(t-r) (D x(r) +p(r)) \, dW(r) \notag \\&+ \int_0^t S(t-r) (D_0 x(r) + p_{0}(r)) \, dW_0(r),
\end{align}
where $p(t)=F_2g(t)+ \sigma,\, p_{0}(t)=F_0g(t)+ \sigma_{0} $, and $g \in C_{\mathcal{F}^0}(\mfT; L^2(\Omega; H)) $, and by the cost functional 
\begin{equation} \label{cosfinlimisincom}
J(u)=\mathbb{E}\int_{0}^{T}\left(\left|M^{\frac{1}{2}}\left (x(t)-\widehat{F}_{1}g(t)  \right ) \right|^{2}+\left| u(t) \right|^{2}\right)dt+\mathbb{E}\left|G^{\frac{1}{2}}\left (x(T)-\widehat{F}_{2}g(T)  \right ) \right|^{2}.
\end{equation} 
\begin{remark}\label{refo} We may rewrite \eqref{dylimisincom} as 
\begin{equation} \label{dylimisincomrf}
x(t) = S(t)\xi + \int_0^t S(t-r)(B u(r)+F_1g(r)) \, dr + \int_0^t S(t-r) (\widebar{D} x(r) +\widebar{p}(r)) \, d\widebar{W}(r), 
\end{equation}
where $\widebar{W}$ is a $V^2$-valued $\widebar{Q}$-Wiener process defined by $\widebar{W}=(W, W_0)$, with the covariance operator $\widebar{Q}=(Q,Q_0)$. The operator $\widebar{D} \in \mathcal{L}(H;\mathcal{L}_2(V_{Q}\times V_{Q_0}; H)) $ is defined as 
\begin{equation} 
(\widebar{D}x)(v_1,v_2)= (Dx)v_1+(D_0x)v_2, \quad\forall x \in H \,\text{and}\,\,\, \forall v_1,v_2 \in V.
\end{equation}
Obviously, we have
\begin{equation}
\left\|\widebar{D} \right\|^2_{\mathcal{L}(H;\mathcal{L}_2(V_{Q}\times V_{Q_0}; H))} \leq  \left\|D \right\|^2_{\mathcal{L}(H;\mathcal{L}_2(V_{Q}; H))}+\left\|D_0 \right\|^2_{\mathcal{L}(H;\mathcal{L}_2(V_{Q_0}; H))}.    
\end{equation}
Similarly, we define $\widebar p(t)\in \mathcal{L}_2(V_{Q}\times V_{Q_0};H)$ as 
\begin{equation}
(\widebar{p}(t))(v_1,v_2)=p(t)v_1+p_0(t)v_2.  \end{equation}
\end{remark}
The next theorem characterizes the optimal control law for the problem described by \eqref{dylimisincom} and \eqref{cosfinlimisincom}.

\begin{theorem}[Optimal Control Law]\label{thm:opt-cntrl-limiting-common-noise} Suppose that \Cref{init-cond-iid} and \Cref{Info-set-Adm-Cntrl-CN} hold. The optimal control law $u^{\circ}$ for the Hilbert space-valued stochastic control problem described by \eqref{dylimisincom}-\eqref{cosfinlimisincom} is given by 
\begin{equation} \label{opticoncom}
u^{\circ}(t) = -B^{\ast}\left(\Pi(t) x(t) -q(t) \right),
\end{equation}
where $\Pi \in C_s(\mfT; \mathcal{L}(H))$, with $\Pi(t)$ being a positive operator for all $t \in \mfT $, satisfies the operator differential Riccati equation 
\begin{equation} \label{riccacom}
 \frac{d}{dt}\left<\Pi(t)x, x\right> = -2\left<\Pi(t)x, Ax\right> +\left<\Pi(t)BB^{\ast}\Pi(t)x, x\right> -\left<D^{\ast}\Pi(t)Dx, x\right>-\left<D^{\ast}_{0}\Pi(t)D_0x, x\right>  - \left<Mx, x\right>,  \forall x \in \mathcal{D}(A),  
\end{equation}
and the pair $(q,\widetilde{q}) \in \mathcal{M}^{2}_{\mathcal{F}^{0}}(\mfT;H) \times \mathcal{M}_{\mathcal{F}^{0}}^2(\mfT;\mathcal{L}_2(V_{Q_0};H))$, satisfies the backward linear stochastic evolution equation
\begin{align} \label{qstcom}
q(t)=&S^{*}(T-t)G\widehat{F}_{2}g(T)-\int_{t}^{T}S^{*}(r-t)\Big(\Pi(r)BB^{\ast}q(r)-M\widehat{F}_{1}g(r)+D^{\ast}\Pi(r)p(r)\\ \notag &+D^{\ast}_{0} (\Pi(r)p_{0}(r)-\widetilde{q}(r))+\Pi(r)F_1g(r)\Big)dr-\int_{t}^{T}S^{*}(r-t)\widetilde{q}(r)dW_0(r).
\end{align}
\end{theorem}
\begin{proof}
We first note that the presence of a common $Q_0$-Wiener process does not affect the operator differential Riccati equation, which therefore remains deterministic. Specifically, \eqref{riccacom} is the backward form of \cite[Eq.~4.21]{liu2025hilbert} with $E=0$. Hence, the arguments for existence and uniqueness presented in the proof of \cite[Thm.~4.3]{liu2025hilbert} apply directly to \eqref{riccacom}. Moreover, the solution to \eqref{qstcom} is a pair $(q,\widetilde{q}) \in \mathcal{M}^{2}_{\mathcal{F}^{0}}(\mfT;H) \times \mathcal{M}_{\mathcal{F}^{0}}^2(\mfT;\mathcal{L}_2(V_{Q_0};H))$.
The existence and uniqueness of such a solution follow from standard results (see, e.g., \cite{hu1991adapted,guatteri2005backward}). Then, following a standard methodology for characterizing the solution to an LQ optimal control, we first aim to apply Itô's lemma to the expression 
\begin{equation}
\langle \Pi(t)x(t), x(t) \rangle - 2 \langle q(t), x(t) \rangle.\label{step1} 
\end{equation}
However, since Itô’s lemma is applicable only to strong solutions of stochastic evolution equations, we instead use the approximating sequence given by 
\begin{equation}
\langle \Pi(t)x_{1,n}(t), x_{1,n}(t) \rangle - 2 \langle q_n(t), x_{2,n}(t) \rangle, 
\end{equation}
 where $x_{1,n}$ and $x_{2,n}$ denote two distinct approximations of the state process $x$ given by 
\begin{align} \label{stateapp1}
x_{1,n}(t) = &\xi + \int_0^t (Ax_{2,n}(r)+J_{n}(B u(r)+F_1g(r))) \, dr + \int_0^t J_{n} (D x_{1,n}(r) +p(r)) \, dW(r) \notag \\
&+ \int_0^t J_{n}(D_0 x_{1,n}(r) + p_{0}(r)) \, dW_0(r),\\
x_{2,n}(t) = &\xi + \int_0^t (A_nx_{2,n}(r)+(B u(r)+F_1 g(r))) \, dr + \int_0^t (D x_{1,n}(r) +p(r)) \, dW(r) \notag \\
&+ \int_0^t (D_0 x_{1,n}(r) + p_{0}(r)) \, dW_0(r),
\end{align}
which are associated with the resolvent operator $J_{n}= n(nI-A)^{-1}$ of $A$ and the Yosida approximation $A_{n}=An(nI-A)^{-1}$ of $A$, respectively. We note that the Riccati equation is deterministic, whereas the offset equation is stochastic. 
Therefore, as can be seen, we treat the corresponding expressions in \eqref{step1} differently.

First, we apply Itô’s lemma to $\langle \Pi(t)x_{1,n}(t), x_{1,n}(t) \rangle$, perform mathematical manipulations, and use the convergence properties of the approximating sequence to derive  
\begin{align} \label{costpi}
&\mathbb{E}\left<\Pi(T)x(T),x(T) \right>=\mathbb{E}\left<\Pi(0)\xi,\xi \right>+\mathbb{E} \int_{0}^{T}\Big[-\left<Mx(t),x(t) \right>+2\left<\Pi(t)x(t),Bu(t)+F_1g(t) \right>\notag \\&+\left<\Pi(t)BB^{\ast }\Pi(t)x(t),x(t) \right>+2\left<D^{\ast}\Pi(t)p(t)+D^{\ast}_{0}\Pi(r)p_{0}(t), x(t)\right>\Big ]dt\notag \\&+\mathbb{E}\int_{0}^{T}\left[\mathrm{tr} \left(\Pi(t)\left( p(t) Q^{1/2} \right)  \left( p(t) Q^{1/2} \right)^{\ast} \right)+\mathrm{tr} \left(\Pi(t)\left( p_0(t) Q_0^{1/2} \right)  \left( p_0(t) Q_0^{1/2} \right)^{\ast} \right)\right]dt.
\end{align} 
Next, we introduce an approximation sequence for \eqref{qstcom} which is a sequence of strong solutions $(q_n(t),\widetilde{q}_n(t))$ such that 
\begin{align} \label{qnstcom}
 q_n(t)=&G\widehat{F}_{2}g(T)+\int_{t}^{T}\big(A_n^{\ast}q_n(r)-\Pi(r)BB^{\ast}q_n(r)+M\widehat{F}_{1}g(r)-D^{\ast}\Pi(r)p(r)\\ \notag &-D^{\ast}_{0} (\Pi(r)p_{0}(r)-\widetilde{q}_{n}(r))-\Pi(r)F_1g(r)\big)dr-\int_{t}^{T}\widetilde{q}_n(r)dW_0(r),
\end{align}
where $A^{\ast}_n=A^{\ast} n(nI-A^{\ast})^{-1}$. From \cite[Theorem 4.4]{guatteri2005backward}, we have the convergence property
\begin{equation} \label{convbsde}
\lim_{n \to \infty }  \sup_{0\leq t\leq T} \mathbb{E}\left|q_{n}(t)-q(t) \right|^{2}=0,\qquad \displaystyle \lim_{n \to \infty }\mathbb{E}\int_{0}^{T}\left\|\widetilde{q}_n(t)-\widetilde{q}(t) \right\|^2 dt=0. 
\end{equation}
Similarly, for \eqref{dylimicom} we have the approximation sequence given by
\begin{align} \label{stateapp}
x_{2,n}(t) = \xi + \int_0^t (A_{n}x_{2,n}(r)+B u(r)+F_1 g(r)) \, dr + \int_0^t  (D x_{2,n}(r) +p(r)) \, dW(r) \notag + \int_0^t (D_0 x_{2,n}(r) + p_{0}(r)) \, dW_0(r),
\end{align}
with
\begin{equation}
 \lim_{n \to \infty }  \sup_{0\leq t\leq T} \mathbb{E}\left|x_{2,n}(t)-x(t) \right|^{2}=0, 
\end{equation}
where the alternative representation given by \eqref{dylimisincomrf} enables the application of existing results for stochastic evolution equations with only one infinite-dimensional noise. Then, we apply Itô's formula to the process \( \langle q_n(t), x_{2,n}(t) \rangle \) over the interval \(\mfT\), as in \cite{guatteri2005backward, brzezniak2008ito}, and take the expectation of both sides to obtain
\begin{align} \label{costq}
 \mathbb{E}\left [ \left<q(T),x(T) \right> \right ]=&\mathbb{E}\left< q(0),\xi \right>+\mathbb{E} \int_{0}^{T}\Big[\Big<x(t),\Pi(t)BB^{\ast}q(t)-M\widehat{F}_{1}g(r)+D^{\ast}\Pi(t)p(t)+D^{\ast}_{0}\Pi(t)p_{0}(t)\notag\\  &+\Pi(t)F_1g(t) \Big> +  \left<q(t),Bu(t)+F_1g(t) \right>+\mathrm{tr} \left(\left( \widetilde{q}(t) Q_0^{1/2} \right)  \left( p_0(t) Q_0^{1/2} \right)^{\ast} \right)\Big]dt.
\end{align}
From \eqref{costpi} and \eqref{costq}, we obtain an expression for $ \mathbb{E} \left [ \left<\Pi(t)x(t),x(t) \right>-2\left<q(t),x(t) \right> \right ]$ which yields
\begin{align} \label{costop}
 J(u) =& \mathbb{E} \left<\Pi(0)\xi,\xi \right> - 2\mathbb{E}\left<q(0),\xi \right> + 2\mathbb{E}\left<G\widehat{F}_{2}g(T), \widehat{F}_{2}g(T) \right>+ \mathbb{E}\Bigg[\int_{0}^{T}\Big|u(t)+B^* \Pi(t)x(t)-B^* q(t) \Big|^{2}dt\Bigg]
 \notag \\ & +\mathbb{E}\int_{0}^{T}\bigg[\left<MF_{1}g(t),F_{1}g(t) \right>-\left |B^{\ast}q(t) \right |^{2}-2\left<q(t),F_1 g(t) \right>-2\mathrm{tr} \left(\left( \widetilde{q}(t) Q_0^{1/2} \right)  \left( p_0(t) Q_0^{1/2} \right)^{\ast} \right)\notag \\ &\quad\quad+\mathrm{tr} \left(\Pi(t)\left( p(t) Q^{1/2} \right)  \left( p(t) Q^{1/2} \right)^{\ast} \right) +\mathrm{tr} \left(\Pi(t)\left( p_0(t) Q_0^{1/2} \right)  \left( p_0(t) Q_0^{1/2} \right)^{\ast} \right)\notag \bigg]dt.
\end{align}
Finally, from the above equation, we derive the optimal feedback control as given by \eqref{opticoncom}.
\end{proof}

\begin{remark}
We note that \eqref{qstcom} may be understood as the mild solution of the backward stochastic evolution equation in Hilbert space given by
\begin{align}
dq(t)=&-\Big(\left(A^{\ast}-\Pi(t)BB^{\ast}\right)q(t)-M\widehat{F}_{1}g(t)+D^{\ast}\Pi(t)p(t)+D^{\ast}_{0} \left(\Pi(t)p_{0}(t)-\widetilde{q}(t)\right)+\Pi(t)F_1g(t)\Big)dt\notag \\ & +\widetilde{q}(t)dW_0(t).
\end{align}
\end{remark}

\subsubsection{Resulting Mean Field Equation}
We aim to characterize the  resulting mean field, denoted by $y_g\in C_{\mathcal{F}^0}(\mfT; L^2(\Omega; H))$, when the representative agent uses the best strategy, given by \eqref{opticoncom}, in response to the fixed mean field \(g \in C_{\mathcal{F}^0}(\mfT; L^2(\Omega; H))\). 

From \Cref{thm:opt-cntrl-limiting-common-noise}, the optimal state $x^\circ(t)$ for the representative agent in the limiting case is given by
\begin{align} \label{eqstate}
x^\circ(t)=&S(t)\xi-\int_{0}^{t}S(t-r)(BB^{\ast}\Pi(r)x^\circ(r)-BB^{\ast}q(r)-F_1g(r))dr\notag \\ &+\int_0^t S(t-r) (D x^\circ(r) + F_2g(r)+\sigma) \, dW(r)+ \int_0^t S(t-r) (D_0 x^\circ(r) + F_{0}g(r)+\sigma_0) \, dW_0(r),
\end{align}
where the process $q$ satisfies \eqref{qstcom}. First, we show that $x^\circ(t)$ is uniformly bounded, and then we use it to derive the mean field equation. Throughout the rest of the paper, the constant \( C \) may vary from one instance to another.
\begin{proposition} \label{prior}
For the optimal state of a representative agent, satisfying \eqref{eqstate}, we have
\begin{equation} \label{boundxi}
 \mathbb{E}\left |x^{\circ}(t) \right | ^2 \leq C,\quad \forall t \in \mfT,
\end{equation}
where the constant $C$ only depends on the model parameters.
\end{proposition}

\begin{proof}
From \eqref{eqstate}, we obtain
\begin{align} \label{boundxieq}
\mathbb{E}\left |x^{\circ}(t) \right | ^2 \leq &   C\Big(\mathbb{E}\left |\xi \right | ^2+ \mathbb{E}\int_{0}^{t}\left |BB^{\ast}\Pi(r)x^{\circ}(r)-BB^{\ast}q(r)-F_1g(r) \right |^{2}dr \notag \allowdisplaybreaks\\ &+\mathbb{E}\int_{0}^{t}\left |D x^{\circ}(r) + F_2g(r)+\sigma \right |^{2}dr+\mathbb{E}\int_{0}^{t}\left |D_0 x^{\circ}(r) + F_{0}g(r)+\sigma_0 \right |^{2}dr \Big) \notag \allowdisplaybreaks\\  \leq & C\Big(\mathbb{E}\left |\xi \right | ^2+ (\left\|B \right\|^{2}+\left\|D \right\|^{2}+\left\|D_0 \right\|^{2}) \mathbb{E}\int_{0}^{t} \left | x^{\circ}(r)\right |^2dr+\left\|B \right\|^{2}\mathbb{E}\int_{0}^{t}\left | q(r)\right |^2dr \notag \allowdisplaybreaks\\ &+(\left\|F_0 \right\|^{2}+\left\|F_1 \right\|^{2}+\left\|F_2\right\|^{2})\mathbb{E}\int_{0}^{t}\left | g(r)\right |^2 dr+\mathbb{E}\int_{0}^{t}(\left\|\sigma \right\|^{2}+\left\|\sigma_0 \right\|^{2})dr \Big)
\notag \allowdisplaybreaks\\  \leq & C(1+ \mathbb{E}\int_{0}^{t}\left | x^{\circ}(r)\right |^2dr).
\end{align}
By applying Grönwall's inequality to the above inequality, we obtain \eqref{boundxi}.
\end{proof}
\begin{proposition} \label{MF-eq}
Suppose that \Cref{init-cond-iid} holds. For a fixed \( g \in C_{\mathcal{F}^0} (\mfT; L^2(\Omega; H))\), we have 
\begin{equation} \label{convergence1}
\lim_{N \rightarrow \infty}\sup_{t \in \mfT} \mathbb{E} \left| y_g(t) - \frac{1}{N}\sum_{i\in \mc{N}}x^{\circ}_{i}(t) \right|_H^2 = 0, 
\end{equation}
where $x_i^\circ(t)$ satisfies
\begin{align} 
x^\circ_i(t)=&S(t)\xi-\int_{0}^{t}S(t-r)(BB^{\ast}\Pi(r)x^\circ_i(r)-BB^{\ast}q(r)-F_1g(r))dr\notag \\ &+\int_0^t S(t-r) (D x^\circ_i(r) + F_2g(r)+\sigma) \, dW_i(r)+ \int_0^t S(t-r) (D_0 x^\circ_i(r) + F_{0}g(r)+\sigma_0) \, dW_0(r),
\end{align}
and $y_g(t)$ is given by 
\begin{align}\label{yq}
y_g(t)=\,& S(t)\bar{\xi} - \int_0^t S(t - r)\left( BB^{\ast}\Pi(r)y_g(r) - BB^{\ast}q(r) - F_1 g(r) \right) dr \notag \\
&+ \int_0^t S(t - r)\left( D_0 y_g(r) + F_0 g(r) + \sigma_0 \right) dW_0(r).
\end{align}

\end{proposition}

\begin{proof}
We omit the proof of \eqref{convergence1}, as it is a special case of \Cref{conatmcom}, which is proved later in the paper and relies on the results of \Cref{prior}.
\end{proof}
\begin{remark} An alternative way to derive the mean field equation is to verify that the well-established results (as used, for instance, in \cite[Lemma~5.1]{carmona2016probabilistic}) also hold in Hilbert spaces, with a few necessary and elementary modifications. Subsequently, by taking conditional expectation $\mathbb{E}[\,\cdot\, | \mathcal{F}^{0}_t]$ on both sides of \eqref{eqstate}, and using \Cref{init-cond-iid}, we obtain
\begin{align}
\mathbb{E}\left[x^\circ(t) | \mathcal{F}^{0}_t\right]
&= S(t)\bar{\xi}
  - \int_{0}^{t} S(t-r) \big( BB^{\ast}\Pi(r)\, \mathbb{E}[ x^\circ(r) | \mathcal{F}^{0}_r ]
      - BB^{\ast} q(r) - F_1 g(r) \big) dr \notag \\
&\quad 
  + \int_{0}^{t} S(t-r) \big( D_0\, \mathbb{E}[ x^\circ(r) | \mathcal{F}^{0}_r ]
      + F_0 g(r) + \sigma_0 \big) \, dW_0(r),
\end{align}
which coincides with \eqref{yq}.
\end{remark}
\subsubsection{Fixed-Point Problem: Existence and Uniqueness for a Small Time Horizon}\label{sec:fixed-point-small-time} We impose the mean-field consistency condition given by  
\begin{equation}
    y_g(t)=g(t), \quad \forall t \in \mfT, \quad \mathbb{P}-a.s.\label{MF-CE}
\end{equation}
Subsequently, we seek a fixed point of the mapping $\Upsilon:g\mapsto y_g$ corresponding to \eqref{MF-CE} in the space
$C_{\mathcal{F}^0}(\mathcal{T}; L^2(\Omega; H))$.
This fixed point characterizes the mean field. To apply Banach's fixed-point theorem,
we establish uniform bounds on $\mathcal T$ for the operators and processes appearing in $\Upsilon$.

By reformulating the state equation as in \eqref{dylimisincomrf} and from \cite[Prop. 4.4]{liu2025hilbert}, we obtain a bound for the operator $\Pi(t)$ as in
\begin{gather}
\left \| \Pi(t) \right \|_{\mathcal{L}(H)} \leq \mathcal{C}_{\Pi}(T),\quad \forall t \in \mfT,\allowdisplaybreaks\\ 
\mathcal{C}_{\Pi}(T):=2M_{T}^{2}\mathrm{exp}\Big((8TM_{T}^{2}(\left\|D\right\|^2+\left\|D_0\right\|^2)\left ( \left \| G \right \|+T\left \| M \right \|  \right )\Big). \label{eq:pi-boundcom}
\end{gather}
We also impose the following assumption for the remainder of this section. Although this assumption may initially appear as a restriction on the model operators and the time horizon, we will demonstrate that it is automatically satisfied under the contraction condition we establish, thus imposing no additional restrictions on the model.  
 \begin{assumption} \label{alT}
 The model operators and the time horizon are such that
 \begin{equation*}
\alpha(T):=16 M^2_{T}T\left\|D_0 \right\|^{2}< 1.     
 \end{equation*}
   
 \end{assumption}
\begin{lemma}[Global Lipschitz Continuity of $q(t)$ wrt $g(t)$]\label{lemma:bounded-variationcom} Suppose that \Cref{alT} holds. Consider the processes $\Pi \in C_s(\mfT; \mathcal{L}(H))$ and $(q,\widetilde{q}) \in \mathcal{M}^{2}_{\mathcal{F}^{0}}(\mfT;H) \times \mathcal{M}_{\mathcal{F}^{0}}^2(\mfT; \mathcal{L}_2(V_{Q_0};H)) $, satisfying \eqref{riccacom} and \eqref{qstcom}, respectively. Then, for any two processes $g_1, g_2 \in  C_{\mathcal{F}^0}(\mathcal{T}; L^2(\Omega; H))$, we have 
\begin{gather}
\mathbb{E}\left |q_1(t)-q_2(t)\right |^{2} \leq \mathcal{C}_1(T) \left|g_1-g_2 \right|^{2}_{C_{\mathcal{F}^0}(\mathcal{T}; L^2(\Omega; H))},\quad \forall t \in \mfT,\label{qboundcom}
\end{gather}
where $q_1$ and $q_2$ are, respectively, the solutions of \eqref{qstcom} corresponding to the inputs $g=g_1$ and $g=g_2$, and 
\begin{align} \label{C1}
\mathcal{C}_1(T)=&\frac{2M^2_{T}}{1-\alpha(T) }\mathrm{exp}\left(\frac{8M^2_{T}}{1-\alpha(T) }\mathcal{C}^{2}_{\Pi}(T)\left\| B\right\|^4\right)\Bigg((\left\| G\right\|\left\| \widehat{F}_2\right\|)^{2}+16T^{2}\Big((\left\| M\right\|\left\| \widehat{F}_1\right\|)^{2}\notag\\&\hspace{5cm}+\mathcal{C}^{2}_{\Pi}(T)\left((\left\| D\right\|\left\| F_2\right\|)^{2}+(\left\| D_0\right\|\left\|F_{0} \right\|)^{2}+\left\| F_1\right\|^{2}\right)\Big)\Bigg).
\end{align}
\end{lemma}
\begin{proof} From \eqref{qstcom}, and for inputs \( g=g_1, g_2 \in C_{\mathcal{F}^0}(\mathcal{T}; L^2(\Omega; H)) \), we have
\begin{align} 
\delta(t)=&S^{*}(T-t)G\widehat{F}_{2}(g_{1}(T)-g_{2}(T))-\int_{t}^{T}S^{*}(r-t)\left(\Pi(r)BB^{\ast}\delta(r)+\psi(r)-D^{\ast}_{0}\widetilde{\delta}(r)\right)dr\notag \\ &-\int_{t}^{T}S^{*}(r-t)\widetilde{\delta}(r)dW_0(r)
\end{align}
where $\delta(t)=q_{1}(t)-q_{2}(t)$, $\widetilde{\delta}(t)=\widetilde{q}_{1}(t)-\widetilde{q}_{2}(t)$, and
\begin{equation}
  \psi(t)= (\Pi(t)F_1-M\widehat{F}_{1})(g_1(t)-g_2(t))+D^{\ast} \big(\Pi(t)F_2(g_1(t)-g_2(t))\big)+D^{\ast}_{0} \big(\Pi(t)F_{0}(g_1(t)-g_2(t))\big).
\end{equation}
Moreover, from \cite[Lemma~2.1]{hu1991adapted}, we obtain
\begin{align} 
\mathbb{E}\left |\delta(t) \right |^2 \leq &\, 2 M^2_{T}T\mathbb{E} \int_{t}^{T}\Big |\Pi(r)BB^{\ast}\delta(r)+\psi(r)-D^{\ast}_{0}\widetilde{\delta}(r) \Big |^2dr+2 (M_{T}\left\|G \right\|\left\|\widehat{F}_2 \right\|)^2\mathbb{E}\left |g_1(T)-g_2(T) \right |^2 \notag \\ \leq &\, 4 M^2_{T}T\mathbb{E} \int_{t}^{T}\Big |\Pi(r)BB^{\ast}\delta(r)+\psi(r) \Big |^2dr+\frac{\alpha(T)}{4}\mathbb{E}\int_{t}^{T}\left\| \widetilde{\delta}(r)\right\|^{2}dr +2 (M_{T}\left\|G \right\|\left\|\widehat{F}_2 \right\|)^2\mathbb{E}\left |g_1(T)-g_2(T) \right |^2,\label{dd} 
\end{align}
and
\begin{align}
\mathbb{E}\int_{t}^{T}\left\|\widetilde{\delta}(r) \right\|^{2}dt &\leq 8 M^2_{T}T\mathbb{E} \int_{t}^{T}\Big |\Pi(r)BB^{\ast}\delta(r)+\psi(r)-D^{\ast}_0\widetilde{\delta}(r)) \Big |^2dr +8 (M_{T}\left\|G \right\|\left\| \widehat{F}_2 \right\|)^2\mathbb{E}\left |g_1(T)-g_2(T) \right |^2 \notag \\ & \leq  16 M^2_{T}T\mathbb{E} \int_{t}^{T}\Big |\Pi(r)BB^{\ast}\delta(r)+\psi(r) \Big |^2dr+\alpha(T)\mathbb{E}\int_{t}^{T}\left\| \widetilde{\delta}(r)\right\|^{2}dr \notag \\ &\hspace{0.4cm}+8 (M_{T}\left\|G \right\|\left\|\widehat{F}_2 \right\|)^2\mathbb{E}\left |g_1(T)-g_2(T) \right |^2.
\end{align}
From \Cref{alT}, we further obtain
\begin{equation} \label{ddt}
  \mathbb{E}\int_{t}^{T}\left\|\widetilde{\delta}(r) \right\|^{2}dt  \leq \frac{1}{1-\alpha(T) } \Big(16 M^2_{T}T\,\mathbb{E} \int_{t}^{T}\Big |\Pi(r)BB^{\ast}\delta(r)+\psi(r) \Big |^2dr+ 8 \big(M_{T}\left\|G \right\|\left\|\widehat{F}_2 \right\|\big)^2\mathbb{E}\left |g_1(T)-g_2(T) \right |^2\Big).
\end{equation}
We then substitute \eqref{ddt} into \eqref{dd} to get 
\begin{align}
\mathbb{E}\left |\delta(t) \right |^2  &\leq  \frac{2M^2_{T}}{1-\alpha (T)}\left[2T\,\mathbb{E} \int_{t}^{T}\Big |\Pi(r)BB^{\ast}\delta(r)+\psi(r) \Big |^2dr+\left\|G \right\|^{2}\left\|\widehat{F}_2 \right\|^2\mathbb{E}\left |g_1(T)-g_2(T) \right |^2\right] \notag \\ &\leq \frac{2M^2_{T}}{1-\alpha(T) }\left[4T\,\mathbb{E} \int_{t}^{T}\Big(\Big |\Pi(r)BB^{\ast}\delta(r)\Big |^2+\left|\psi(r) \right|^2\Big)dr+\left\|G \right\|^{2}\left\|\widehat{F}_2 \right\|^2\mathbb{E}\left|g_1(T)-g_2(T) \right |^2\right] \notag \\  &\leq \frac{2M^2_{T}}{1-\alpha (T)}\bigg[4T\,\mathcal{C}^{2}_{\Pi}(T)\left\| B\right\|^4 \int_{t}^{T}\mathbb{E}\left |\delta(r) \right |^2dr+\Bigg(16T^{2}\Big((\left\| M\right\|\left\| \widehat{F}_1\right\|)^{2}+(\mathcal{C}_{\Pi}(T)\left\| F_1\right\|)^{2}\notag\\&\hspace{0.15cm}+\mathcal{C}^{2}_{\Pi}(T)\left((\left\| D\right\|\left\| F_2\right\|)^{2}+(\left\| D_0\right\|\left\|F_{0} \right\|)^{2}\right)\Big)+(\left\| G\right\|\left\| \widehat{F}_2\right\|)^{2}\Bigg)\left |g_1-g_2 \right |^{2}_{C_{\mathcal{F}^0}(\mathcal{T}; L^2(\Omega; H))} \bigg].
\end{align}
Finally, by the application of Grönwall's inequality, \eqref{qboundcom} follows. 
\end{proof}

\begin{theorem}[Contraction Condition]\label{Contraction-thm}
Suppose that \Cref{alT} holds. The mapping 
\[
\Upsilon: g \in C_{\mathcal{F}^0}(\mathcal{T}; L^2(\Omega; H)) \longrightarrow y_g \in C_{\mathcal{F}^0}(\mathcal{T}; L^2(\Omega; H)),
\]
associated with the mean-field consistency condition \eqref{MF-CE} admits a unique fixed point, denoted by $\bar{x}\in C_{\mathcal{F}^0}(\mathcal{T}; L^2(\Omega; H))$, if 
\begin{equation} \label{contrcom}
   \mathcal{C}_2(T) e^{T \mathcal{C}_3(T)} < 1,
\end{equation}
where
\begin{align} \label{C23}
   \mathcal{C}_2(T) &= 5M_T^2 T \left[ T\left( \|B\|^4 \mathcal{C}_1(T) + \|F_1\|^2 \right) + \|F_0\|^2 \right], \notag \\
   \mathcal{C}_3(T) &= 5M_T^2 \left[\|D_0\|^2 + T \|B\|^4 \mathcal{C}_\Pi^2(T) \right].
\end{align}
\end{theorem}
\begin{proof}
From \eqref{yq}, we have
\begin{align}
y_{g_1}(t) - y_{g_2}(t) &= -\int_{0}^{t} S(t - r) BB^{\ast} \Pi(r) \left( y_{g_1}(r) - y_{g_2}(r) \right) dr \notag \\
&\quad + \int_0^t S(t - r) D_0 \left( y_{g_1}(r) - y_{g_2}(r) \right) dW_0(r) + \int_0^t S(t - r) F_1 \left( g_1(r) - g_2(r) \right) dr \notag \\
&\quad + \int_0^t S(t - r) F_0 \left( g_{1}(r) - g_{2}(r) \right) dW_0(r)  + \int_0^t S(t - r) BB^{\ast} \left( q_1(r) - q_2(r) \right) dr \notag \\
&= \mathcal{I}_1 + \mathcal{I}_2 + \mathcal{I}_3 + \mathcal{I}_4 + \mathcal{I}_5.
\end{align}

Subsequently, we obtain
\begin{equation}
\mathbb{E} \left| y_{g_1}(t) - y_{g_2}(t) \right|^2 
\leq 5 \left( \mathbb{E} \left| \mathcal{I}_1 \right|^2 + \mathbb{E} \left| \mathcal{I}_2 \right|^2 + \mathbb{E} \left| \mathcal{I}_3 \right|^2 + \mathbb{E} \left| \mathcal{I}_4 \right|^2 + \mathbb{E} \left| \mathcal{I}_5 \right|^2 \right).
\end{equation}

For \( \mathbb{E} \left| \mathcal{I}_5 \right|^2 \), using \Cref{lemma:bounded-variationcom}, we have
\begin{align}
\mathbb{E} \left| \mathcal{I}_5 \right|^2 
&= \mathbb{E} \left| \int_{0}^{t} S(t - r) BB^{\ast} (q_1(r) - q_2(r)) \, dr \right|^2 \leq T \, \mathbb{E} \int_{0}^{t} \left| S(t - r) BB^{\ast} (q_1(r) - q_2(r)) \right|^2 dr \notag\\&\leq M_T^2 T \|B\|^4 \int_{0}^{T}\mathbb{E}\left |q_1(t)-q_2(t)\right |^{2}dt  \leq M_T^2 T^2 \|B\|^4 \mathcal{C}_1(T) \left| g_1 - g_2 \right|^2_{C_{\mathcal{F}^0}(\mathcal{T}; L^2(\Omega; H))}.
\end{align}
Using a similar treatment for the remaining terms and collecting all contributions, we obtain
\begin{equation}
\mathbb{E} \left| y_{g_1}(t) - y_{g_2}(t) \right|^2 
\leq \mathcal{C}_2(T) \left| g_1 - g_2 \right|^2_{C_{\mathcal{F}^0}(\mathcal{T}; L^2(\Omega; H))} 
+ \mathcal{C}_3(T) \int_0^t \mathbb{E} \left| y_{g_1}(r) - y_{g_2}(r) \right|^2 dr.
\end{equation}
The result then follows from Grönwall's inequality, yielding  
\begin{equation}
\mathbb{E} \left| y_{g_1}(t) - y_{g_2}(t) \right|^2 
 \leq \mathcal{C}_2(T)\, \exp(T\mathcal{C}_3(T))\left| g_1 - g_2 \right|^2_{C_{\mathcal{F}^0}(\mathcal{T}; L^2(\Omega; H))},
\end{equation}
and thus 
\begin{equation*}
 \left| y_{g_1} - y_{g_2}\right|_{_{C_{\mathcal{F}^0}(\mathcal{T}; L^2(\Omega; H))}}
 \leq (\mathcal{C}_2(T)\, \exp(T\mathcal{C}_3(T)))^\frac{1}{2} \left| g_1 - g_2 \right|_{C_{\mathcal{F}^0}(\mathcal{T}; L^2(\Omega; H))}.  
\end{equation*}
\end{proof}
\begin{remark} \label{contrcomre}
It is straightforward to verify the following convergence properties as \( T \rightarrow 0 \): (i) \( M_T \rightarrow 1 \),
    (ii) \( \alpha(T) \rightarrow 0 \),
    (iii) \( \mathcal{C}_{\Pi}(T) \rightarrow 1 \),
    (iv) \( \mathcal{C}_{1}(T) \rightarrow \theta > 0 \),
     (v) \( \mathcal{C}_{2}(T) \rightarrow 0 \),
    (vi) \( T \mathcal{C}_{3}(T) \rightarrow 0 \). Moreover, all of these functions are continuous in \( T \). Therefore, it can be shown that the contraction condition \eqref{contrcom} is satisfied for a sufficiently small time horizon \( T > 0 \). Moreover, for such a small time horizon \Cref{alT} is satisfied as we have \( \alpha(T) \rightarrow 0 \) as \( T \rightarrow 0 \).  Thus, \Cref{alT} does not impose any further restriction on the model parameters. 
\end{remark}
Suppose that the contraction condition \eqref{contrcom} holds, the unique fixed point \( \bar{x} \in C_{\mathcal{F}^0}(\mathcal{T}; L^2(\Omega; H)) \) of the mapping $\Upsilon$ associated with the mean-field consistency condition \eqref{MF-CE} characterizes the mean field.

\subsubsection{Nash Equilibrium} \label{nashcomsec}
\begin{theorem}[Nash Equilibrium]\label{Nash-eqcom}Consider the Hilbert space-valued limiting system, described by \eqref{dylimi-S0com} and \eqref{cosfinlimi-S0com} for $i\in\mathbb{N}$, and suppose that \Cref{init-cond-iid}, \Cref{Info-set-Adm-Cntrl-CN}, and the contraction condition \eqref{contrcom} hold. Then, the set of control laws $\{u_{i}^{\circ}\}_{i \in \mb{N}}$, where $u_{i}^{\circ}$ is given by
\begin{equation} \label{nashcom}
u_{i}^{\circ}=-B^{\ast}\left(\Pi(t) x_i(t) -q(t) \right),
\end{equation}
forms a unique Nash equilibrium for the limiting system. These control laws are fully characterized by the fixed point consisting of the mean field $\bar{x}(t) \in H$, the operator $\Pi(t) \in \mc{L}(H)$ and the pair $(q(t), \widetilde{q}(t))\in H \times \mathcal{L}_2(V_{Q_0},H)$, which solves the set of mean-field consistency equations given by 
\begin{align}
 & \frac{d}{dt}\left<\Pi(t)x, x\right> = -2\left<\Pi(t)x, Ax\right> +\left<\Pi(t)BB^{\ast}\Pi(t)x, x\right>\notag \allowdisplaybreaks\\&\hspace{2.4cm}-\left<D^{\ast}\Pi(t)Dx, x\right>-\left<D^{\ast}_{0}\Pi(t)D_{0}x, x\right>  - \left<Mx, x\right>, \quad    \Pi(T) = G,\quad x \in \mathcal{D}(A),\label{MF-eq-1}  \allowdisplaybreaks\\
 &q(t)=S^{*}(T-t)G\widehat{F}_{2}\bar{x}(T)-\int_{t}^{T}S^{*}(r-t)\Big(\Pi(r)BB^{\ast}q(r)-M\widehat{F}_{1}\bar{x}(r)+D^{\ast}(\Pi(r)(F_2\bar{x}(r)+ \sigma))\notag \allowdisplaybreaks\\  &\hspace{1cm}+D^{\ast}_{0} (\Pi(r)(F_0\bar{x}(r)+ \sigma_{0})-\widetilde{q}(r))+\Pi(r)F_1\bar{x}(r)\Big)dr-\int_{t}^{T}S^{*}(r-t)\widetilde{q}(r)dW_0(r), \allowdisplaybreaks\\
 &\bar{x}(t)=S(t)\bar{\xi}-\int_{0}^{t}S(t-r)((BB^{\ast}\Pi(r)-F_1)\bar{x}(r)-BB^{\ast}q(r))dr\notag \allowdisplaybreaks\\ &\hspace{1cm}+ \int_0^t S(t-r) ((D_0 + F_{0})\bar{x}(r)+\sigma_0) \, dW_0(r). \label{mfcom}
\end{align}
\end{theorem}
\begin{proof}
The proof is a direct result of \Cref{thm:opt-cntrl-limiting-common-noise}, \Cref{MF-eq} and \Cref{Contraction-thm}. We note that the equilibrium state in \eqref{nashcom} is given by 
\begin{align} \label{eqstatecom}
x_i(t)=&S(t)\xi_i-\int_{0}^{t}S(t-r)(BB^{\ast}\Pi(r)x_i(r)-BB^{\ast}q(r)-F_1\bar{x}(r))dr\notag \\ &+\int_0^t S(t-r) (D x_i(r) + F_2\bar{x}(r)+\sigma) \, dW_i(r)+ \int_0^t S(t-r) (D_0 x_i(r) + F_{0}\bar{x}(r)+\sigma_0) \, dW_0(r).
\end{align}
Moreover, \eqref{mfcom}  is the mild solution to the stochastic differential equation given by
\begin{align}
d\bar{x}(t)=(A\bar{x}(t)-((BB^{\ast}\Pi(t)-F_1)\bar{x}(t)+BB^{\ast}q(t))dt+ ((D_0 + F_{0})\bar{x}(t)+\sigma_0)dW_0(t),
\end{align}
with $\bar{x}(0)=\bar{\xi}$.
\end{proof}
\subsection{{$\varepsilon$-Nash Property}}\label{sec:epnashcom}
In this section, we establish the $\epsilon$-Nash property of the set of control laws $\{u_i^{\circ}\}_{i \in \mc{N}}$ specified in \eqref{nashcom} for the $N$-player game defined by \eqref{statecom}--\eqref{cosfinNcom}. Due to the symmetry (exchangeability) among agents, it suffices to consider the case in which agent $i = 1$ deviates from these Nash equilibrium strategies. Specifically, we assume that all agents $i \in \mc{N}$, $i \neq 1$, adopt the feedback strategy $u_i^{[N],\circ}$ given by 
\begin{equation} \label{equil-strategy-N-player}
u^{[N],\circ}_i(t) = -B^{\ast}\left( \Pi(t)\, x^{[N]}_i(t)- q(t)\right),
\end{equation}
while agent $i=1$ is allowed to employ an arbitrary strategy $u_1^{[N]} \in \mc{U}^{[N],0}$. Here, for clarity we use the superscript $[N]$ to denote the processes associated with the $N$-player game. The resulting system of stochastic evolution equations is then given by
\begin{align} 
x_{1}^{[N]}(t)=&S(t)\xi_1+\int_{0}^{t}S(t-r)\left (Bu_1^{[N]}(r)+F_1x^{(N)}(r)  \right )dr+\int_{0}^{t}S(t-r)p_{1}^{[N]}(r)dW_{1}(r)\notag \allowdisplaybreaks\\&+\int_{0}^{t}S(t-r)p_{1,0}^{[N]}(r)dW_{0}(r),\label{fini1com} \allowdisplaybreaks\\
x_{i}^{[N]}(t)=&S(t)\xi_i+\int_{0}^{t}S(t-r)\left (Bu^{[N],\circ}_{i}(r)+F_1x^{(N)}(r)  \right )dr+\int_{0}^{t}S(t-r)p_{i}^{[N]}(r)dW_{i}(r)\notag \allowdisplaybreaks\\&+\int_{0}^{t}S(t-r)p_{i,0}^{[N]}(r)dW_{0}(r),\quad i\in \mathcal{N},\,\,\text{and}\,\,\, i \neq 1, \label{finiicom}
\end{align}
where $x^{(N)}_t=\frac{1}{N}\sum_{i\in \mc{N}}x^{[N]}_{i}(t)$,\, $p_{i}^{[N]}(t):=Dx_{i}^{[N]}(t)+Fx^{(N)}(t)+\sigma$,\, $p_{i,0}^{[N]}(t):=D_0x_{i}^{[N]}(t)+F_0x^{(N)}(t)+\sigma_0$. 
Furthermore, we recall that the cost functional of agent $i=1$ in the $N$-player game is given by  
\begin{equation}\label{cost-epscom}
 J^{[N]}_1(u_{1}^{[N]},u_{-1}^{[N],\circ}):= \mathbb{E}\int_{0}^{T}\!\!\left(\left|M^{\frac{1}{2}}\left (x_{1}^{[N]}(t)- x^{(N)}(t)  \right ) \right|^{2}+\left| u_{1}^{[N]}(t) \right|^{2}\right)dt+\mathbb{E}\left|G^{\frac{1}{2}}\left (x_{1}^{[N]}(T)- x^{(N)}(T)  \right ) \right|^{2}.
\end{equation}
At equilibrium, where agent $i = 1$ also employs the strategy $u_1^{[N],\circ}$ given by \eqref{equil-strategy-N-player} for $i=1$, we denote its state by $x_1^{[N],\circ}$, the state of agent $i$ by $x_i^{[N],\circ}$, and the corresponding average state of all $N$ agents by $x^{(N),\circ}$.

The $\epsilon$-Nash property indicates that 
\begin{equation} \label{epnashcom}
  J^{[N]}_1(u^{[N],\circ}_{1},u^{[N],\circ}_{-1}) \leq   \inf_{u_1^{[N]} \in \mc{U}^{[N],0}}J^{[N]}_1(u_{1}^{[N]},u^{[N],\circ}_{-1})+\epsilon_N,
\end{equation}
where the sequence  $\{\epsilon_N\}_{N\in \mb{N}}$ converges to zero. The above property can equivalently be expressed as 
\begin{equation} \label{epnashcom1}
  J^{[N]}_1(u^{[N],\circ}_{1},u^{[N],\circ}_{-1}) \leq   \inf_{u_1^{[N]} \in \mc{A}^{[N],0}}J^{[N]}_1(u_{1}^{[N]},u^{[N],\circ}_{-1})+\epsilon_N,
\end{equation}
where $\mc{A}^{[N],0}:=\left\{u_1^{[N]} \in \mc{U}^{[N],0}: J^{[N]}_1(u_1^{[N]},u^{[N],\circ}_{-1}) \leq  J^{[N]}_1(u^{[N],\circ}_{1},u^{[N],\circ}_{-1})  \right\}$, which is a non-empty set. Therefore, in the analysis of this section, we restrict $u_1^{[N]}$ to the set $\mc{A}^{[N],0}$.
\begin{lemma}\label{thm:lemmasumcom}
Consider the $N$ coupled systems described by \eqref{fini1com}--\eqref{finiicom}. Then, the property
\begin{equation}\label{prop-boundednesscom}
\mathbb{E}\left [\sum_{i\in \mc{N}}\left|x_{i}^{[N]}(t) \right|^2_H  \right ] \leq CN ,  
\end{equation}
holds uniformly for all \(t \in \mfT\). Here, the constant \( C \) is independent of both the strategy of agent \( i = 1 \), i.e., \( u_1^{[N]} \in \mathcal{A}^{[N],0} \), and the number of agents \( N \).
\end{lemma}
\begin{proof} 
We first consider the equilibrium case in which agent $i = 1$ also employs $u^{[N],\circ}_1$. In this case, the state of agent $i$, $i\in \mc{N}$, satisfies \eqref{finiicom} with 
$u_i^{[N],\circ}(t) = -B^{\ast}\!\left(\Pi(t)x^{[N],\circ}_i(t) - q(t)\right)$. Therefore, 
\begin{align} \label{boundicom}
\mathbb{E}\left| x_{i}^{[N],\circ}(t)\right|^{2}& \!\leq C^{\circ} \mathbb{E}\left[\left | \xi_i \right |^2\!+\!\!\int_{0}^{t}\left(\left|-BB^{\ast}\left(\Pi(r) x^{[N],\circ}_i(r) -q(r) \right) \!+\!F_1x^{(N),\circ}(r)\right|^{2}\!\!\!+\!\left\|p_{i}^{[N],\circ}\right\|_{\mathcal{L}_2}^{2}+\left\|p_{i,0}^{[N],\circ}\right\|_{\mathcal{L}_2}^{2}\right)dr\right] \notag \\  &\leq  C^{\circ} \Big(\int_{0}^{t}\mathbb{E}\left|x_{i}^{[N],\circ}(r)\right|^{2}dr+\int_{0}^{t}\mathbb{E}\left|x^{(N),\circ}(r)\right|^{2}dr+1\Big) \notag \\  &\leq C^{\circ} \Big(\int_{0}^{t}\mathbb{E}\left|x_{i}^{[N],\circ}(r)\right|^{2}dr+\frac{1}{N}\int_{0}^{t}\mathbb{E}\bigg [\sum_{j\in \mc{N}}\left|x_{j}^{[N],\circ}(r) \right|^{2}  \bigg ]dr+1\Big).
\end{align}
By summing \eqref{boundicom} over all $i \in \mc{N}$, we obtain
\begin{equation}
\mathbb{E}\left [\sum_{i\in \mc{N}}\left|x_{i}^{[N],\circ}(t) \right|^{2}\right ] \leq 
C^{\circ} \left (N+2\int_{0}^{t}\mathbb{E}\left [\sum_{i\in \mc{N}}\left|x_{i}^{[N],\circ}(r) \right|^{2}  \right ]dr  \right ).
\end{equation}
Applying Grönwall's inequality to the above equation results in
\begin{equation}
\mathbb{E}\left [\sum_{i\in \mc{N}}\left|x_{i}^{[N],\circ}(t) \right|^2_H  \right ] \leq C^{\circ}N.     
\end{equation}
Since all agents are symmetric at  equilibrium, for every $t \in \mfT$ and $N \in \mathbb{N}$, we further have 
\begin{equation} \label{steq}
\mathbb{E}\left|x_{i}^{[N],\circ}(t)\right|^{2} \leq C^{\circ}, 
\,\,\, \forall i \in \mc{N},\quad \text{and} \quad  \quad\mathbb{E}\left|x^{(N),\circ}(t)\right|^{2} \leq C^{\circ}.\end{equation}
Then, it is straightforward to verify that for any $u_1^{[N]} \in \mc{A}^{[N],0} $ we have 
\begin{equation} \label{enbound}
\mathbb{E}\left [\int_{0}^{T}\left |u_1^{[N]}(t) \right |^2dt  \right ]\leq   J^{[N]}_1(u_1^{[N]},u^{[N],\circ}_{-1}) \leq J^{[N]}_1(u^{[N],\circ}_{1},u^{[N],\circ}_{-1}) \leq C^{\circ \circ},
\end{equation}
where $C^{\circ \circ}$ is a constant independent of $N$.
Now suppose agent $i=1$ chooses an alternative control $u_1^{[N]} \in  \mc{A}^{[N],0}$. Using \eqref{enbound}, we obtain
\begin{align} \label{bound1com}
\mathbb{E}\left|x_{1}^{[N]}(t) \right|^{2} & \leq C^{\ast} \mathbb{E}\left[\left | \xi_1 \right |^2\!+\!\!\int_{0}^{t}\left(\left|u_{1}^{[N]}(r)\right|^2+\left|F_1 x^{(N)}(r)\right|^{2}\!\!\!+\!\left\|p_{1}^{[N]}\right\|_{\mathcal{L}_2}^{2}+\left\|p_{1,0}^{[N]}\right\|_{\mathcal{L}_2}^{2}\right)dr\right] \notag \\ & \leq C^{\ast} \Big( \int_{0}^{t}\mathbb{E}\left|x_{1}^{[N]}(r)\right|^{2}dr+\frac{1}{N}\int_{0}^{t}\mathbb{E}\Big [\sum_{j\in \mc{N}}\left|x_{j}^{[N]}(r) \right|^{2}  \Big]dr+1\Big), 
\end{align}
where the constant $C^{\ast}$ depends only on $C^{\circ \circ}$ and the model parameters, and is independent of $N$. Moreover, for agent $i$, $i\in \mc{N}$ and $i\neq 1$, we have
\begin{align} \label{boundicom1}
\mathbb{E}\left| x_{i}^{[N]}(t)\right|^{2}& \!\leq C^{\circ} \mathbb{E}\left[\left | \xi_i \right |^2\!+\!\!\int_{0}^{t}\left(\left| -BB^{\ast}\left(\Pi(r) x^{[N]}_i(r) -q(r)\right)\!+\!F_1x^{(N)}(r)\right|^{2}\!\!\!+\!\left\|p_{i}^{[N]}\right\|_{\mathcal{L}_2}^{2}+\left\|p_{i,0}^{[N]}\right\|_{\mathcal{L}_2}^{2}\right)dr\right] \notag \\  &\leq  C^{\circ} \Big(\int_{0}^{t}\mathbb{E}\left|x_{i}^{[N]}(r)\right|^{2}dr+\int_{0}^{t}\mathbb{E}\left|x^{(N)}(r)\right|^{2}dr+1\Big) \notag \\  &\leq C^{\circ} \Big(\int_{0}^{t}\mathbb{E}\left|x_{i}^{[N]}(r)\right|^{2}dr+\frac{1}{N}\int_{0}^{t}\mathbb{E}\bigg [\sum_{j\in \mc{N}}\left|x_{j}^{[N]}(r) \right|^{2}  \bigg ]dr+1\Big).
\end{align}
From \eqref{bound1com} and \eqref{boundicom1}, and by choosing $C = C^{\ast} \vee C^{\circ}$, we obtain
\begin{equation}
\mathbb{E}\left [\sum_{i\in \mc{N}}\left|x_{i}^{[N]}(t) \right|^{2}\right ] \leq 
C\left (N+2\int_{0}^{t}\mathbb{E}\left [\sum_{i\in \mc{N}}\left|x_{i}^{[N]}(r) \right|^{2}  \right ]dr  \right )  .
\end{equation}
Finally, applying Grönwall's inequality to the above equation results in \eqref{prop-boundednesscom}.
\end{proof}
From Lemma~\ref{thm:lemmasumcom}, for every $t\in\mfT$ and $N\in\mathbb{N}$, we have
\begin{equation}\label{boundcom}
\mathbb{E}\left|x^{(N)}(t)\right|^{2}
\leq C.
\end{equation}
Moreover, by \eqref{boundicom1} and \eqref{prop-boundednesscom}, it follows that
\begin{equation}\label{boundcom-ind}
\mathbb{E}\left|x^{[N]}_{i}(t)\right|^{2} \leq C,
\qquad \forall\, i \in \mc{N}.
\end{equation}
\begin{theorem}[Average State Error Bound] \label{conatmcom}
Suppose the state of any agent $i$, $i\in\mc{N}$ and $i\neq 1$, satisfies \eqref{finiicom}, where the agent employs the strategy $u^{[N],\circ}_i$ given by \eqref{equil-strategy-N-player}. For any strategy $u_1^{[N]} \in \mc{A}^{[N],0}$ that agent $i=1$ chooses, we have 
\begin{equation}\label{convergencecom}
\displaystyle \sup_{t\in \mfT}\mathbb{E}\left|\bar{x}(t)- x^{(N)}(t) \right|_{H}^{2} \leq \frac{C}{N}.
\end{equation}
\end{theorem}

\begin{proof}
The average state \( x^{(N)}(t) \), by direct computation, is given by
\begin{align}
x^{(N)}(t) =\,& S(t)x^{(N)}(0) - \int_0^t S(t - r)(BB^{\ast} \Pi(r) - F_1)x^{(N)}(r) \, dr 
+ \int_0^t S(t - r) BB^{\ast} q(r) \, dr \notag \\
&+ \frac{1}{N} \left[ \sum_{i \in \mc{N}} \int_0^t S(t - r) p_i^{[N]}(r) \, dW_i(r) \right]
+ \int_0^t S(t - r) \left( (D_0 + F_0)x^{(N)}(r) + \sigma_0 \right) \, dW_0(r) \notag \\
&+ \frac{1}{N} \int_0^t S(t - r) B \left( u_1^{[N]}(r) + B^{\ast} \Pi\, x_1^{[N]}(r) - B^{\ast} q(r) \right) \, dr.\label{average-state-with-deviation}
\end{align}

Now, define \( y(t) := \bar{x}(t) - x^{(N)}(t) \). Then, from \eqref{mfcom} and \eqref{average-state-with-deviation}, we have  
\begin{align}
y(t) =\,& S(t)y(0) - \int_0^t S(t - r)(BB^{\ast} \Pi(r) - F_1)y(r) \, dr 
+ \int_0^t S(t - r)(D_0 + F_0)y(r) \, dW_0(r) \notag \\
& - \frac{1}{N} \left[ \sum_{i \in \mc{N}} \int_0^t S(t - r) p_i^{[N]}(r) \, dW_i(r) \right] 
- \frac{1}{N} \int_0^t S(t - r) B \left( u_1^{[N]}(r) + B^{\ast}\Pi\, x_1^{[N]}(r) - B^{\ast} q(r) \right) dr.
\end{align}
Subsequently, we obtain the  estimate given by 
\begin{align} \label{procecom}
\mathbb{E} \left| y(t) \right|^2 
\leq\, & \frac{C}{N}+C \int_0^t \mathbb{E} \left| y(r) \right|^2 \, dr 
+ \frac{C}{N^2}  \mathbb{E} \left| \sum_{i \in \mc{N}} \int_0^t S(t - r) p_i^{[N]}(r) \, dW_i(r) \right|^2 \notag \\
&+ \frac{C}{N^2} \left( \int_0^t \mathbb{E} \left| B \left( u_1^{[N]}(r) + B^{\ast} x_1^{[N]}(r) - B^{\ast} q(r) \right) \right|^2 \, dr \right).
\end{align}
Further, we estimate the stochastic fluctuation term as in
\begin{align}
\mathbb{E} \left| \sum_{i \in \mc{N}} \int_0^t S(t - r) p_i^{[N]}(r) \, dW_i(r) \right|^2 
&= \sum_{i \in \mc{N}} \mathbb{E} \left| \int_0^t S(t - r) \left( D x_i^{[N]}(r) + F_2 x^{(N)}(r) + \sigma \right) \, dW_i(r) \right|^2 \notag \\
&\leq C \left( \int_0^t \sum_{i \in \mc{N}} \mathbb{E} \left| x_i^{[N]}(r) \right|^2 \, dr 
+ N \int_0^t \mathbb{E} \left| x^{(N)}(r) \right|^2 \, dr + N \right) \notag \\
&\leq C N,
\end{align}
where the second inequality follows from \eqref{prop-boundednesscom} and \eqref{boundcom}. Moreover, we have
\begin{align}
\int_{0}^{t}\mathbb{E}\left|B\Big(u_1^{[N]}(r)+B^{\ast}x^{[N]}_{1}(r)-B^{\ast}q(r)\Big)\right| ^2dr \leq C\int_{0}^{t}(\mathbb{E}\left|x_{1}^{[N]}(r)\right|^{2}+1)dr \leq C, \label{term4com}
\end{align}
where the first inequality follows from \eqref{enbound} and the properties of $q$ as shown in \cite{hu1991adapted, guatteri2005backward}.  Subsequently, from \eqref{procecom}--\eqref{term4com}, we obtain 
\begin{equation}
\mathbb{E}\left|y(t) \right|^{2} \leq C(\frac{1}{N}+\frac{1}{N^2})+ C\int_{0}^{t}\mathbb{E}\left|y(r)\right|^{2}dr.
\end{equation}
Finally, by applying Grönwall's inequality, \eqref{convergencecom} follows.
\end{proof}
\begin{proposition}[Error Bounds for Agent $i=1$)] \label{cconatmcom}
Let $x_1(t)$ and $x_1^{[N]}(t)$, respectively, denote the state of agent $i=1$ in the limiting game and the $N$-player game satisfying \eqref{dylimi-S0com} and \eqref{fini1com}. Moreover, let $J^{\infty }(.)$ and $J^{[N]}(\,.\,,u^{[N],\circ}_{-1})$, respectively, denote the cost functional of agent $i=1$ in the limiting game and the $N$-player game given by \eqref{cosfinlimi-S0com} and \eqref{cost-epscom}. 
\begin{itemize}
    \item[(i)] If agent $i=1$ employs the strategies $u^\circ_1$ and $u^{[N],\circ}_1$, given by \eqref{nashcom} and \eqref{equil-strategy-N-player}, in the limiting and the $N$-player game, respectively, then we have
\begin{gather}
\displaystyle \sup_{t\in \mfT}\mathbb{E}\left|x_1^{\circ}(t)-{x}_1^{[N],\circ}(t)\right|_{H}^{2} \leq \frac{C}{N},\label{state-conv-1com}\allowdisplaybreaks\\  
\left|J^{\infty }_1(u^{\circ}_{1})-J^{[N]}_1(u^{[N],\circ }_{1},u^{[N],\circ}_{-1}) \right|   \leq \frac{C}{\sqrt{N}}.\label{cost-conv-1com} 
\end{gather}
\item[(ii)] If agent $i=1$ employs any $u_1^{[N]} \in \mc{A}^{[N],0}$ in both the limiting and the $N$-player games, then we have 
\begin{gather}\label{state-conv-2com}
\displaystyle \sup_{t\in \mfT}\mathbb{E}\left|x_{1}(t)-x_{1}^{[N]}(t)\right|_{H}^{2} \leq \frac{C}{N},\allowdisplaybreaks\\
\left|J^{\infty }_1(u_1^{[N]})-J^{[N]}_1(u_1^{[N]},u^{[N],\circ}_{-1}) \right|   \leq \frac{C}{\sqrt{N}}.\label{cost-conv-2com}
\end{gather}
\end{itemize}
\end{proposition}
\begin{proof}
For the case where agent $i =1$ employs strategies $u^\circ_1$ and $u^{[N],\circ}_1$, given by \eqref{nashcom} and \eqref{equil-strategy-N-player}, in the limiting and the $N$-player game, respectively, by direct computation we have 
\begin{align}
x_1^{\circ}(t)-x_{1}^{[N],\circ}(t)&=- \int_{0}^{t}S(t-r)BB^{\ast}\Pi(r)(x_1^{\circ}(r)-x_{1}^{[N],\circ}(r))dr\,+\int_{0}^{t}S(t-r)F_1\left (\bar{x}(r)-{x}^{(N),\circ}(r) \right )dr\notag\allowdisplaybreaks\\
& +\int_{0}^{t}S(t-r)\Big(D(x_1^{\circ}(r)-{x}_1^{[N],\circ}(r))+F_2(\bar{x}(r)-{x}^{(N),\circ}(r) )\Big)dW_1(r)\notag\allowdisplaybreaks\\ &\,+\int_{0}^{t}S(t-r)\Big(D_0(x_1^{\circ}(r)-{x}_1^{[N],\circ}(r))+F_0(\bar{x}(r)-{x}^{(N),\circ}(r) )\Big)dW_0(r).
\end{align}
Moreover, for the case where agent $i=1$ employs an arbitrary strategy $u_1^{[N]} \in \mc{A}^{[N],0}$ in both the limiting and the $N$-player games, we have
\begin{align}
x_{1}(t)-x^{[N]}_1(t)=&\int_{0}^{t}S(t-r)F_1\left (\bar{x}(r)-x^{(N)}(r)  \right )dr\notag\allowdisplaybreaks\\
& +\int_{0}^{t}S(t-r)\Big(D(x_1(r)-{x}_1^{[N]}(r))+F_2(\bar{x}(r)-{x}^{(N)}(r) )\Big)dW_1(r)\notag\allowdisplaybreaks\\ &\,+\int_{0}^{t}S(t-r)\Big(D_0(x_1(r)-{x}_1^{[N]}(r))+F_0(\bar{x}(r)-{x}^{(N)}(r) )\Big)dW_0(r).  \end{align}
The remainder of the proof of \eqref{state-conv-1com} and \eqref{state-conv-2com} follows from computing norms and leveraging the Cauchy-Schwarz inequality and the results from \Cref{conatmcom}, and applying Grönwall’s inequality, similar to the steps used in the proof of \Cref{conatmcom}.

For the property \eqref{cost-conv-1com}, we have
\begin{align}
 \left|J^{\infty }_1(u^{\circ }_{1})-J^{[N]}_1(u^{[N],\circ }_{1},u^{[N],\circ}_{-1}) \right| \leq& \mathbb{E}\int_{0}^{T}\left| \left|M^{ \frac{1}{2}}(x^{\circ}_1(t)-\bar{x}(t)) \right|^2 - \left|M^{ \frac{1}{2}}(x_{1}^{[N],\circ}(t)-x^{(N),\circ}(t)) \right|^2\right|dt \notag \allowdisplaybreaks\\
& + \mathbb{E}\left|\left|G^{ \frac{1}{2}}(x^{\circ}_1(T)-\bar{x}(T)) \right|^2 - \left|G^{ \frac{1}{2}}\left(x_{1}^{[N],\circ}(T)-x^{(N),\circ}(T)\right) \right|^2\right|\nonumber\\
&+\mathbb{E}\int_{0}^{T}\left|\left|B^{\ast}\left( \Pi(t)\, x^{\circ}_1(t)- q(t)\right) \right|^2 - \left|B^{\ast}\left( \Pi(t)\, x^{[N],\circ}_1(t)- q(t)\right) \right|^2\right|dt.
\end{align}
Moreover, for $t \in [0,T)$, we have
\begin{align}   
 \bigg \vert \left|M^{ \frac{1}{2}}(x^{\circ}_1(t)-\bar{x}(t)) \right|^2-&\left|M^{ \frac{1}{2}}(x_{1}^{[N],\circ}(t)-x^{(N),\circ}(t)) \right|^2 \bigg\vert \leq 2\left|M^{ \frac{1}{2}}(x^{\circ}_1(t)-x_{1}^{[N],\circ}(t)) \right|^2\notag\\ & +2\left|M^{ \frac{1}{2}}(\bar{x}(t)-x^{(N),\circ}(t)) \right|^2\notag + 2\left|M^{ \frac{1}{2}}(x^{\circ}_1(t)-\bar{x}(t))\right|\allowdisplaybreaks\\ 
&\,\,\,\, \times \left (2\left|M^{ \frac{1}{2}}(x^{\circ}_1(t)-x_{1}^{[N],\circ}(t)) \right|^2+2\left|M^{ \frac{1}{2}}(\bar{x}(t)-x^{(N),\circ}(t)) \right|^2  \right )^{\frac{1}{2}}.
\end{align}
We apply the same method to the terminal cost and use \eqref{state-conv-1com}, \eqref{convergencecom} and \eqref{boundxi} to establish the estimate \eqref{cost-conv-1com}. Moreover, for any admissible strategy $u_1^{[N]} \in \mc{A}^{[N],0}$ we have 
\begin{align}
 \left|J^{\infty}_1(u_{1}^{[N]})-J^{[N]}_1(u_{1}^{[N]},u^{[N],\circ}_{-1}) \right| \leq& \mathbb{E}\int_{0}^{T}\left| \left|M^{ \frac{1}{2}}\big(x_1(t)-\bar{x}(t)\big) \right|^2 - \left|M^{ \frac{1}{2}}\Big(x_{1}^{[N]}(t)-x^{(N)}(t)\Big) \right|^2\right|dt \notag \allowdisplaybreaks\\
&+ \mathbb{E}\left|\left|G^{ \frac{1}{2}}\big(x_1(T)-\bar{x}(T)\big) \right|^2 - \left|G^{ \frac{1}{2}}\Big(x_{1}^{[N]}(T)-x^{(N)}(T)\Big) \right|^2\right|.
\end{align}
The proof of \eqref{cost-conv-2com} follows by applying the same steps as in the proof of \eqref{cost-conv-1com} to the above inequality.
\end{proof}
\begin{theorem} ($\epsilon$-Nash Equilibrium)\label{E-Nash-thm} Suppose that \Cref{init-cond-iid}, \Cref{clcom} and the contraction condition  \eqref{contrcom} hold. Then, the set of strategies $\{u^{\circ}_i\}_{i \in \mc{N}}$, where  $u^{\circ}_i$ is given by \eqref{nashcom}, forms an $\epsilon$-Nash equilibrium for the $N$-player game described by \eqref{statecom}-\eqref{cosfinNcom}; that is, 
\begin{equation} 
J^{[N]}_1(u^{[N],\circ}_{1},u^{[N],\circ}_{-1}) \leq  \inf_{u_1^{[N]}\in \mc A^{[N],0}} J^{[N]}_1(u^{[N]}_1,u^{[N],\circ}_{-1})+\epsilon_N,
\end{equation} 
where $\epsilon_N =O(\tfrac{1}{\sqrt{N}})$.
\end{theorem}
\begin{proof}
We recall that
\begin{equation} \label{ineinfcom}
 J^{\infty }_1(u^{\circ}_1) \leq J^{\infty }_1(u_1^{[N]}),\quad \forall\, u_1^{[N]} \in\mc{A}^{[N],0}.
\end{equation}
Then, from \eqref{cost-conv-1com}, \eqref{cost-conv-2com} and \eqref{ineinfcom}, we have
\begin{equation}
J^{[N]}_1(u^{[N],\circ }_{1},u^{[N],\circ}_{-1})-\frac{C}{\sqrt{N}} \leq  J^{\infty}_1(u^{\circ}_1) \leq     J^{\infty}_1(u_1^{[N]}) \leq J^{[N]}_1(u_1^{[N]},u^{[N],\circ}_{-1})+\frac{C}{\sqrt{N}},\quad \forall\, u_1^{[N]} \in\mc{A}^{[N],0},
\end{equation}
which gives $J^{[N]}_1(u^{[N],\circ }_{1},u^{[N],\circ}_{-1}) \leq J^{[N]}_1(u_1^{[N]},u^{[N],\circ}_{-1})+2\frac{C}{\sqrt{N}},$  $\forall u_1^{[N]} \in\mc{A}^{[N],0}$.
\end{proof}

\section{Fixed-Point Problem: Existence and Uniqueness for an Arbitrary Finite Time Horizon} \label{arbit-time} 
In this section, we propose a method to solve the fixed-point problem described in \Cref{sec:fixed-point-small-time} for an arbitrary finite time horizon for the case where the diffusion term associated with the common noise in the dynamics of the representative agent, given by \eqref{dylimi-S0com}, is deterministic. This is formalized in the following assumption. 
\begin{assumption}\label{det-diff}
In the \( N \)-player game model \eqref{statecom}–\eqref{cosfinNcom}, and hence in the corresponding limiting model \eqref{dylimi-S0com}–\eqref{cosfinlimi-S0com} as well as the mean-field consistency equations \eqref{MF-eq-1}–\eqref{mfcom}, the coefficients of the diffusion term associated with the common noise satisfy the conditions that (i) $D_0=0$ and (ii) $F_0=0$.
\end{assumption} Specifically, under \Cref{det-diff} and for an arbitrary finite time horizon $\mathcal{T}=[0,T]$, we investigate the existence and uniqueness of a fixed-point solution to the mean-field consistency equations given by 
\begin{align} 
& \frac{d}{dt}\left<\Pi(t)x, x\right> = -2\left<\Pi(t)x, Ax\right> +\left<\Pi(t)BB^{\ast}\Pi(t)x, x\right>\notag\\
&\hspace{5cm}-\left<D^{\ast}\Pi(t)Dx, x\right>- \left<Mx, x\right>,\quad  
\Pi(T) = G,\, \,\, x \in \mathcal{D}(A), \label{Riccati-AT} \allowdisplaybreaks\\
 &q(t)=S^{*}(T-t)G\widehat{F}_{2}\bar{x}(T)-\int_{t}^{T}S^{*}(r-t)\Big(\Pi(r)BB^{\ast}q(r)+\Lambda(r)\bar{x}(r)+D^{*}\sigma\Big)dr \notag \allowdisplaybreaks\\
 &\hspace{9.1cm}-\int_{t}^{T}S^{*}(r-t)\widetilde{q}(r)dW_0(r), \label{Offset-AT} \\
 &\bar{x}(t)=S(t)\bar{\xi}-\int_{0}^{t}S(t-r)((BB^{\ast}\Pi(r)-F_1)\bar{x}(r)-BB^{\ast}q(r))dr + \int_0^t S(t-r) \sigma_0 \, dW_0(r),\label{Mean-Field-AT}
\end{align}
where $\Lambda := D^{*}\Pi F_{2}+\Pi F_{1}-M\widehat{F}_{1}$ belongs to $C_{s}(\mathcal{T}; \mathcal{L}(H))$. We first present the existence results, followed by those concerning uniqueness. Finally, we discuss the required assumptions.
\subsection{Existence Results}
Our approach to establishing existence takes inspiration from \cite{Linear-FBSDES-Yong-1999} 
and develops the analysis for coupled stochastic evolution equations in Hilbert spaces, 
where only mild (rather than strong) solutions exist. 
This methodology differs from that of \cite{federico2024linear}, which deals with a class of deterministic evolution equations. Specifically, we establish existence results for a broader class of LQ MFGs that also incorporate common noise, leading to a system of coupled linear forward-backward stochastic evolution equations (FBSEEs) in Hilbert spaces, given by \eqref{Offset-AT}--\eqref{Mean-Field-AT}, through the identification of a decoupling field. To this end, we impose the following assumptions.
\begin{assumption} \label{as61}
The family of operator-valued Riccati equations 
\begin{align} \label{etan}
&\dot\eta_n(t)
\;+\;\eta_n(t)\big(A_n - BB^{*}\Pi(t)\big)
\;+\;\big(A_n^{*} - \Pi(t)BB^{*}\big)\eta_n(t) \;+\;\eta_n(t)F_1
\;-\;\eta_n(t)BB^{*}\eta_n(t)
\;+\;\Lambda(t)
\;=\;0,\nonumber\\
&
\eta_n(T)=-G\,\widehat F_2 ,
\end{align}
where $n\in\mathbb{N}$ and $A_n=An(A-nI)^{-1}$ denotes the Yosida approximation of the operator $A$, admits strict solutions \(\eta_n \in C_{s}\!\big(\mathcal{T};\mathcal{L}(H)\big)\),
with \(\eta_n(t)\) uniformly bounded for all \(t \in \mathcal{T}\) and for all \(n \in \mathbb{N}\).
 \end{assumption} 
\begin{assumption} \label{as62}
There exists \(\eta \in C_{s}\!\big(\mathcal{T};\mathcal{L}(H)\big)\) such that, for every \(t\in\mathcal{T}\), \(\eta_n(t)\) converges strongly to \(\eta(t)\) in \(\mathcal{L}(H)\); that is,
\begin{equation} \label{etaneta}
\left|\eta_n(t)x-\eta(t)x\right|_{H}\to 0,
\qquad \forall x\in H,\quad \forall t\in\mathcal{T},
\quad \text{as } n\to\infty .
\end{equation}
\end{assumption}
 \begin{assumption} \label{as62}
There exists \(\eta \in C_{s}\!\big(\mathcal{T};\mathcal{L}(H)\big)\), such that $\eta_n$ converges to $\eta$ in $C_{s}\!\big(\mathcal{T};\mathcal{L}(H)\big)$; that is,
\begin{equation} \label{etaneta}
\sup_{t\in \mathcal{T}} \left |\eta_n(t)x- \eta(t)x\right |_{H}\to 0, \,\,\, \forall x \in H,\quad \text{as}\quad n \rightarrow \infty.
\end{equation}
 \end{assumption} 
\begin{proposition}
Suppose that  \Cref{as61} and \Cref{as62} hold.
\begin{itemize}
    \item[(i)] \cite{hu1991adapted,guatteri2005backward} The backward stochastic differential equations given by
\begin{align}
&d\varsigma_n(t)
= -\Big( A_n^{*}\varsigma_n(t) - \Pi(t)BB^{*}\varsigma_n(t) 
   - \eta_n(t)BB^{*}\varsigma_n(t) - D^{*}\sigma \Big)\,dt + \;\widetilde\varsigma_n(t)\,dW_0(t), 
\quad \varsigma_n(T) = 0, \label{vsign}\\
&d\varsigma(t)
= -\Big( A^{*}\varsigma(t) - \Pi(t)BB^{*}\varsigma(t) 
   - \eta(t)BB^{*}\varsigma(t) - D^{*}\sigma \Big)\,dt  + \;\widetilde\varsigma(t)\,dW_0(t), 
\qquad\qquad \varsigma(T) = 0, \label{vsig}
\end{align}
admit a unique pair of mild solutions $\big(\varsigma_n, \widetilde\varsigma_n\big), \big(\varsigma, \widetilde\varsigma\big) \in \mathcal{M}_{\mathcal{F}^0}^2(\mathcal{T};H) \times \mathcal{M}_{\mathcal{F}^{0}}^2(\mfT; \mathcal{L}_2(V_{Q_0},H))$, respectively, satisfying the uniform bounds 
\begin{align} \label{vsibound}
&\sup_{t\in\mathcal{T}}\mathbb{E}\,|\varsigma(t)|^{2}
\;+\;\mathbb{E}\!\int_{0}^{T}\!\|\widetilde\varsigma(t)\|_{\mathcal L_{2}}^{2}\,dt
\;\leq\; C, \\[4pt]
&\sup_{t\in\mathcal{T}}\mathbb{E}\,|\varsigma_n(t)|^{2}
\;+\;\mathbb{E}\!\int_{0}^{T}\!\|\widetilde\varsigma_n(t)\|_{\mathcal L_{2}}^{2}\,dt
\;\leq\; C, \qquad \forall\,n \in \mathbb{N}.\label{vsi-n-bound}
\end{align}

\item[(ii)] The following convergence results hold.
\begin{equation} \label{kavp3}
\lim_{n \to \infty } \sup_{t\in\mathcal{T}} \mathbb{E}\,|\varsigma_{n}(t)-\varsigma(t)|^{2}=0, 
\qquad 
\lim_{n \to \infty }\mathbb{E}\int_{0}^{T}\!\|\widetilde{\varsigma}_n(t)-\widetilde{\varsigma}(t)\|^{2}\,dt=0.
\end{equation}
\end{itemize}
\end{proposition}
\begin{proof}
The first part of the proposition concerns standard results on the existence, uniqueness, and boundedness of solutions to backward stochastic evolution equations in Hilbert spaces \cite{hu1991adapted,guatteri2005backward}, as already applied to equations~\eqref{qstcom} and~\eqref{qnstcom}. 

For the second part, we first introduce another sequence of backward stochastic equations given by
\begin{equation} \label{kappan}
\kappa_n(t)
= - \int_{t}^{T} S_n^{*}(r-t)\Big(
   \Pi(r)BB^{*}\kappa_n(r)
   + \eta(r)BB^{*}\varsigma(r)
   + D^{\ast} \sigma \Big)\,dr  - \int_{t}^{T} S_n^{*}(r-t)\,\widetilde{\kappa}_n(r)\,dW_0(r),
\end{equation}
where $t \in \mathcal{T}$. Under \Cref{as62} and due to the well-posedness of $\varsigma$ as the solution to \eqref{vsig}, equation \eqref{kappan} admits a unique mild (strong) solution $\big(\kappa_n,
\widetilde\kappa_n \big)\in \mathcal{M}^{2}_{\mathcal{F}^0}(\mathcal{T};H) \times  \mathcal{M}_{\mathcal{F}^{0}}^2(\mfT; \mathcal{L}_2(V_{Q_0},H))$. Moreover, since \eqref{kappan} represents the Yosida approximation sequence of \eqref{vsig} (see e.g. \cite{guatteri2005backward}), 
by the same reasoning as in \eqref{qstcom} and \eqref{qnstcom}, we obtain
\begin{equation} \label{kavp}
\lim_{n \to \infty } \sup_{t\in\mathcal{T}} \mathbb{E}\,|\kappa_{n}(t)-\varsigma(t)|^{2} = 0, 
\qquad 
\lim_{n \to \infty } \mathbb{E}\!\int_{0}^{T}\!\|\widetilde{\kappa}_n(t)-\widetilde{\varsigma}(t)\|^{2}\,dt = 0.
\end{equation}
Now, we observe that the difference sequence $\varsigma_n(t) - \kappa_n(t)$ satisfies
\begin{align}
\varsigma_n(t)-\kappa_n(t)
&= - \int_{t}^{T} S_n^{*}(r-t)\Big(
   \Pi(r)BB^{*}\big(\varsigma_n(r)-\kappa_n(r)\big)
   + \big(\eta_n(r)-\eta(r)\big)BB^{*}\varsigma(r) \Big)\,dr \nonumber \\
&\hspace{6.89cm} - \int_{t}^{T} S_n^{*}(r-t)\,\big(\widetilde{\varsigma}_n(r)-\widetilde{\kappa}_n(r)\big)\,dW_0(r),
\end{align}
for all $t\in\mathcal{T}$. By \cite[Lemma~2.1]{hu1991adapted}, we have 
\begin{align}
&\mathbb{E}\!\left|\varsigma_n(t)-\kappa_n(t)\right|^2 
\leq 2 M_T^{2} T\,\mathbb{E}\!\int_{t}^{T}
\Big|
\Pi(r)BB^{*}\big(\varsigma_n(r)-\kappa_n(r)\big)
+ \big(\eta_n(r)-\eta(r)\big)BB^{*}\varsigma(r)
\Big|^{2}dr, \label{knvn} \\[6pt]
&\mathbb{E}\!\int_{0}^{T}\!\|\widetilde{\kappa}_n(t)-\widetilde{\varsigma}_n(t)\|^{2}\,dt 
\leq 8 M_T^{2} T\,\mathbb{E}\!\int_{0}^{T}
\Big|
\Pi(r)BB^{*}\big(\varsigma_n(r)-\kappa_n(r)\big)
+ \big(\eta_n(r)-\eta(r)\big)BB^{*}\varsigma(r)
\Big|^{2}dr. \label{whknvn}
\end{align}
From \eqref{knvn}, we further obtain 
\begin{equation}
\mathbb{E}\!\left|\varsigma_n(t)-\kappa_n(t)\right|^{2}
\leq 
C\,\mathbb{E}\!\int_{0}^{T}
\Big|
\big(\eta_n(r)-\eta(r)\big)BB^{*}\varsigma(r)
\Big|^{2}dr
+ 
C\,\mathbb{E}\!\int_{t}^{T}
\left|\varsigma_n(r)-\kappa_n(r)\right|^{2}dr,
\end{equation}
which, after applying Grönwall’s inequality, yields
\begin{equation}
\sup_{t\in\mathcal{T}} 
\mathbb{E}\!\left|\varsigma_n(t)-\kappa_n(t)\right|^{2}
\leq 
C\,\mathbb{E}\!\int_{0}^{T}
\Big|
\big(\eta_n(r)-\eta(r)\big)BB^{*}\varsigma(r)
\Big|^{2}dr,
\end{equation}
where $C>0$ is a constant independent of $n$. Based on \Cref{as62} and the dominated convergence theorem, we have 
\begin{equation} \label{kavp1}
\lim_{n \to \infty }\mathbb{E}\!\int_{0}^{T}
\Big|
\big(\eta_n(r)-\eta(r)\big)BB^{*}\varsigma(r)
\Big|^{2}dr=0.
\end{equation}
Consequently, we obtain
\begin{equation}
\lim_{n \to \infty }\sup_{t\in\mathcal{T}} 
\mathbb{E}\!\left|\varsigma_n(t)-\kappa_n(t)\right|^{2}=0.
\end{equation}
It follows immediately from \eqref{whknvn} that 
\begin{equation} \label{kavp2}
\lim_{n \to \infty }\mathbb{E}\!\int_{0}^{T}
\|\widetilde{\kappa}_n(t)-\widetilde{\varsigma}_n(t)\|^{2}\,dt=0.
\end{equation}
Finally, \eqref{kavp3} follows by combining \eqref{kavp}, \eqref{kavp1} and \eqref{kavp2}.
\end{proof}
 Now, we investigate the mean-field consistency equations given by \eqref{Riccati-AT}--\eqref{Mean-Field-AT}. Note that \eqref{Riccati-AT} can be solved independently of \eqref{Offset-AT}--\eqref{Mean-Field-AT} over an arbitrary time horizon, as demonstrated in the proof of \Cref{thm:opt-cntrl-limiting-common-noise}. However, the latter equations are coupled both with each other and with \eqref{Riccati-AT}. Hence, in the remainder of this section, we focus on the solution of the FBSEE \eqref{Offset-AT}–\eqref{Mean-Field-AT}. First, we introduce the following Lemma regarding an approximating sequence for these equations.
\begin{lemma} \label{soappr}
Suppose that \Cref{det-diff} and \Cref{as61} hold. Then, for any $T > 0$ and $n\in\mathbb{N}$, the approximating system corresponding to the FBSEE \eqref{Offset-AT}--\eqref{Mean-Field-AT} given by  
\begin{align} 
d\bar x_{n}(t)
&=\Big(A_{n}\bar x_{n}(t) - \big(BB^{*}\Pi(t)-F_1\big)\bar x_{n}(t) + BB^{*}q_{n}(t)\Big)\,dt
   + \sigma_0\,dW_0(t),\,\,\,\,\qquad \bar x_{n}(0) = \bar\xi, \label{barxn} \\
dq_{n}(t)
&=-\Big(A_{n}^{*}q_{n}(t) - \Pi(t)\,BB^{*}\,q_{n}(t) - \Lambda(t)\,\bar x_{n}(t) - D^{*}\sigma\Big)\,dt
   + \widetilde q_{n}(t)\,dW_0(t),\,\,\,\,
q_{n}(T) = G\,\widehat F_2\,\bar x_{n}(T),  \label{apqn}
\end{align}
admits at least one mild (strong) solution  ($\bar{x}_n,  (q_n,\widetilde q_n)$), where  
\begin{equation} \label{qnwqn}
q_n(t):=-\eta_n(t)\,\bar x_n(t)+\varsigma_n(t), 
\qquad 
\widetilde q_n(t):=\widetilde\varsigma_n(t)-\eta_n(t)\,\sigma_0,
\end{equation}
and $\eta_n$ and  $(\varsigma_n,\widetilde\varsigma_n)$  satisfy \eqref{etan} and \eqref{vsign}, respectively.
\end{lemma}
\begin{proof}
Suppose that $\bar{x}_n$ is the unique mild (strong) solution of the stochastic evolution equation 
\begin{equation}\label{MF-sub}
 d\bar x_{n}(t)
 =\Big(A_{n}\bar x_{n}(t) - \big(BB^{*}\Pi(t)-F_1\big)\bar x_{n}(t) - BB^{*}\big(\eta_n(t)\,\bar x_n(t)-\varsigma_n(t)\big)\Big)\,dt
   + \sigma_0\,dW_0(t),
\end{equation}
which is obtained by substituting \eqref{qnwqn} in \eqref{barxn}. Since $\bar x_{n}$ and $\varsigma_n$ are the strong solutions of \eqref{MF-sub} and \eqref{vsign}, respectively, we can apply Itô's formula to 
\begin{equation*}
 q_n(t):=-\eta_n(t)\,\bar x_n(t)+\varsigma_n(t),  
\end{equation*}
and use \eqref{etan} to obtain
\begin{equation}
dq_{n}(t)
=-\Big(A_{n}^{*}q_{n}(t) - \Pi(t)\,BB^{*}\,q_{n}(t) - \Lambda(t)\,\bar x_{n}(t) - D^{*}\sigma\Big)\,dt 
+ \big(\widetilde\varsigma_n(t)-\eta_n(t)\,\sigma_0\big)\,dW_0(t),
\end{equation}
which coincides with \eqref{apqn} when $\widetilde q_n(t)=\widetilde\varsigma_n(t)-\eta_n(t)\,\sigma_0$.
\end{proof}
\begin{theorem} \label{soah}
Suppose that \Cref{det-diff}, \Cref{as61} and \Cref{as62} hold. Then, for any $T > 0$, the FBSEE \eqref{Offset-AT}-\eqref{Mean-Field-AT} admits at least one mild solution.
\end{theorem}
\begin{proof}
From \Cref{soappr}, the solution of the FBSEE \eqref{barxn}--\eqref{apqn} can be equivalently expressed as the following (decoupled) stochastic equations 
\begin{align} 
&\bar{x}_n(t)=S_n(t)\bar{\xi}-\int_{0}^{t}S_n(t-r)\big(\Phi_n(r)\bar{x}_n(r)-BB^{\ast}\varsigma_n(r)\big)dr+ \int_0^t S_n(t-r) \sigma_0 \, dW_0(r), \label{xnmild}\\
 &q_n(t)=S^{*}_n(T-t)G\widehat{F}_{2}\bar{x}_{n}(T)-\int_{t}^{T}S^{*}_n(r-t)\Big(\Pi(r)BB^{\ast}q_n(r)+\Lambda(r)\bar{x}_n(r)+D^{\ast} \sigma\Big)dr\notag \\  &\hspace{10.5cm}-\int_{t}^{T}S^{*}_n(r-t)\widetilde{q}_n(r)dW_0(r), \label{sys6ap}
 \end{align}
where $n \in \mathbb{N}$, and $\Phi_n(t)=BB^{\ast}(\eta_n(t)+\Pi(t))-F_1$, which results from substituting \eqref{qnwqn} into \eqref{barxn}. We note that \eqref{xnmild} is decoupled from \eqref{sys6ap}. Moreover, consider the stochastic evolution equation
\begin{equation} \label{barx}
\bar{x}(t)
= S(t)\bar{\xi}
- \int_{0}^{t} S(t-r)\big(\Phi(r)\bar{x}(r) - BB^{*}\varsigma(r)\big)\,dr
+ \int_{0}^{t} S(t-r)\sigma_0\,dW_0(r),
\end{equation}
where $\Phi(t) = BB^{*}\big(\eta(t) + \Pi(t)\big) - F_1$ and $\varsigma$ is the solution to the well-posed equation \eqref{vsig}.  The above stochastic evolution equation admits a unique mild solution (see \cite{hu1991adapted, guatteri2005backward}.). Moreover, based on \Cref{as61} and \Cref{as62}, and from \eqref{vsibound} and \eqref{vsi-n-bound}, we have
\begin{equation}
\mathbb{E}|\bar{x}_n(t)|^2 \leq C, \qquad \mathbb{E}|\bar{x}(t)|^2 \leq C,\quad \forall n \in \mathbb{N}, \quad t \in \mathcal{T}.
\end{equation} 

Now, we show that \eqref{xnmild} converges to \eqref{barx} in mean square by following the steps below. From \eqref{barx} and \eqref{xnmild}, we have 
\begin{equation}
 \bar{x}(t)-\bar{x}_n(t)=\rho_n(t)-\int_{0}^{t}S_n(t-r)\Phi_n(r)(\bar{x}(r)-\bar{x}_n(r))dr,
\end{equation}
where, for every $t \in \mathcal{T}$, $\rho_n(t)$ is given by
\begin{align}
\rho_n(t)
&= \big(S(t) - S_n(t)\big)\,\bar{\xi}  - \int_{0}^{t} \big(S(t-r) - S_n(t-r)\big)
   \big(\Phi(r)\,\bar{x}(r) - BB^{*}\varsigma(r)\big)\,dr  - \int_{0}^{t} S_n(t-r)\nonumber \allowdisplaybreaks\\
&\quad \times\Big(
   \big(\eta(r) - \eta_n(r)\big)\,\bar{x}(r)
   - BB^{*}\big(\varsigma(r) - \varsigma_n(r)\big)\Big)\,dr  + \int_{0}^{t} \big(S(t-r) - S_n(t-r)\big)\,\sigma_0\,dW_0(r).
\end{align}
Thus,
\begin{equation}
\mathbb{E}|\bar{x}(t)-\bar{x}_n(t)|^2
\le 2\,\mathbb{E}|\rho_n(t)|^2
   + C \int_0^t \mathbb{E}|\bar{x}(s)-\bar{x}_n(s)|^2\,ds,
   \qquad t\in\mathcal{T}. \label{ineq_gron}
\end{equation}
It is straightforward to verify that for all $t \in \mathcal{T}$
\begin{equation} \label{rhon}
  \mathbb{E}|\rho_n(t)|^2 \to 0\,\,\,\qquad \text{as}\,\,\, n \rightarrow \infty.
\end{equation}
Then, by Grönwall’s inequality, for each fixed \(t \in \mathcal{T}\), we have 
\begin{equation}
\mathbb{E}|\bar{x}(t)-\bar{x}_n(t)|^2
\le
\,C \mathbb{E}|\rho_n(t)|^2.
 \label{ineq_gron_exp}
 \end{equation}
Therefore, we obtain
\begin{equation} \label{barxncx}
  \lim_{n \to \infty }   \mathbb{E}|\bar{x}(t)-\bar{x}_n(t)|^2=0,
\qquad \forall t\in\mathcal{T}.
\end{equation}
Drawing ideas from \Cref{soappr}, for every $t\in\mathcal{T}$, we  define 
\begin{equation} \label{conqn}
q(t) \;:=\; -\,\eta(t)\,\bar{x}(t) \,+\, \varsigma(t),\qquad  \widetilde q(t) \;:=\; \widetilde\varsigma(t) \,-\, \eta(t)\,\sigma_0,
\end{equation}
and substitute these expressions into \eqref{barx}. After rearranging the terms, we obtain 
\begin{equation}\label{conqn1}
 \bar{x}(t)=S(t)\bar{\xi}-\int_{0}^{t}S(t-r)((BB^{\ast}\Pi(r)-F_1)\bar{x}(r)-BB^{\ast}q(r))dr + \int_0^t S(t-r) \sigma_0 \, dW_0(r).   
\end{equation}
We now aim to show that \eqref{conqn1} coincides with \eqref{Mean-Field-AT}, which implies that the latter admits a unique solution. To this end, we show that the term $q$ appearing in \eqref{conqn1} satisfies \eqref{Offset-AT}.

From \eqref{etaneta}, \eqref{kavp3} and \eqref{barxncx}, it can be shown that \eqref{qnwqn} converges to \eqref{conqn} in mean square. Specifically, $\forall t \in \mathcal{T}$,
\begin{align}
&\lim_{n \to \infty }  \mathbb{E}\left|q_n(t)-q(t) \right|^2=0, \label{conqnlh} \\   
&\lim_{n \to \infty }\mathbb{E}\int_{0}^{T}\!\|\widetilde{q}_n(t)-\widetilde{q}(t)\|^{2}\,dt=0.
\end{align}
We now show that the right-hand side of \eqref{sys6ap} converges, in mean square, to that of \eqref{Offset-AT}. When rearranged, the difference between these expressions is given by
\begin{align} \label{conqnrh}
&\Big[S^{*}(T-t)G\widehat{F}_{2}\bar{x}(T)
 -\int_{t}^{T}S^{*}(r-t)\Big(\Pi(r)BB^{\ast}q(r)+\Lambda(r)\bar{x}(r)+D^{*}\sigma\Big)\,dr
 -\int_{t}^{T}S^{*}(r-t)\widetilde{q}(r)\,dW_0(r)\Big] \nonumber\allowdisplaybreaks\\
&\,-\Big[S^{*}_n(T-t)G\widehat{F}_{2}\bar{x}_n(T)
 -\int_{t}^{T}S^{*}_n(r-t)\Big(\Pi(r)BB^{\ast}q_n(r)+\Lambda(r)\bar{x}_n(r)+D^{*}\sigma\Big)\,dr
 \nonumber\allowdisplaybreaks\\&\qquad-\int_{t}^{T}S^{*}_n(r-t)\widetilde{q}_n(r)\,dW_0(r)\Big] \nonumber\allowdisplaybreaks\\
=& \big(S^{*}(T-t)-S^{*}_n(T-t)\big)G\widehat{F}_{2}\bar{x}(T)+S^{*}_n(T-t)G\widehat{F}_{2}\big(\bar{x}(T)-\bar{x}_n(T)\big) \notag\allowdisplaybreaks\\
&- \int_{t}^{T}\big(S^{*}(r-t)-S^{*}_n(r-t)\big)
     \Big(\Pi(r)BB^{\ast}q(r)+\Lambda(r)\bar{x}(r)+D^{*}\sigma\Big)\,dr \nonumber\allowdisplaybreaks\\
&- \int_{t}^{T} S^{*}_n(r-t)\Big(\Pi(r)BB^{\ast}\big(q(r)-q_n(r)\big)
        + \Lambda(r)\big(\bar{x}(r)-\bar{x}_n(r)\big)\Big)\,dr \notag\allowdisplaybreaks\\
        &- \int_{t}^{T} S^{*}_{n}(r-t)\big(\widetilde{q}(r)-\widetilde{q}_n(r)\big)\,dW_0(r)- \int_{t}^{T} \big(S^{*}(r-t)-S^{*}_n(r-t)\big)\widetilde{q}(r)\,dW_0(r), 
\end{align}
which converges to zero in $L^2(\Omega;H)$ based on the dominated convergence theorem and the previously established convergence results. Up to now, we showed that for every $t\in\mfT$: (i) $q_n(t)=-\eta_n(t)\,\bar x_n(t)+\varsigma_n(t)$ and $\widetilde q_n(t)=\widetilde\varsigma_n(t)+\eta_n(t)\,\sigma_0,$ converge to $q(t) \;=\; -\,\eta(t)\,\bar{x}(t) \,+\, \varsigma(t)$ and $\widetilde q(t) \;=\; \widetilde\varsigma(t) \,+\, \eta(t)\,\sigma_0$, respectively, (ii) the right-hand-side of the stochastic evolution equation satisfied by $q_n(t)$ converges to that of \eqref{Offset-AT}. Hence, we conclude that for each $t \in \mfT$, $q(t) \;:=\; -\,\eta(t)\,\bar{x}(t) \,+\, \varsigma(t)$ satisfies  
\begin{equation} \label{qt11}
q(t)=S^{*}(T-t)G\widehat{F}_{2}\bar{x}(T)-\int_{t}^{T}S^{*}(r-t)\Big(\Pi(r)BB^{\ast}q(r)+\Lambda(r)\bar{x}(r)+D^{*}\sigma\Big)dr-\int_{t}^{T}S^{*}(r-t)\widetilde{q}(r)dW_0(r), a.s, 
\end{equation}
which coincides with \eqref{Offset-AT}. As $\bar{x}$
  satisfies the well-posed equation \eqref{barx} and is unique, the above stochastic evolution equation is well-defined and admits a unique solution (see \cite{hu1991adapted, guatteri2005backward}). The desired result follows from \eqref{barx} and \eqref{qt11} and the well-posedness of both equations. Hence, by adopting the decoupling field $q(t) \;:=\; -\,\eta(t)\,\bar{x}(t) \,+\, \varsigma(t),\,t\in\mathcal{T}$, we have demonstrated that there exists at least one solution to the FBSEE \eqref{Offset-AT}--\eqref{Mean-Field-AT}.
\end{proof}
\subsection{Uniqueness Results}
We have established the existence of a solution to the FBSEEs given by \eqref{Offset-AT}--\eqref{Mean-Field-AT}. 
We now turn to the uniqueness of the solution, taking inspiration from the proof of \cite[Thm.~4.2]{federico2024linear}. 
To this end, we first impose the following additional assumption.
\begin{assumption} 
\label{ass63} 
In the \( N \)-player game model \eqref{statecom}–\eqref{cosfinNcom}, and hence in the corresponding limiting game \eqref{dylimi-S0com}--\eqref{cosfinlimi-S0com} as well as the mean-field consistency equations \eqref{MF-eq-1}--\eqref{mfcom}, the following conditions hold.
\begin{itemize}
    \item[(i)] $F_1=0$; that is, there is no mean field coupling in the drift coefficient of the agents’ dynamics.
    \item[(ii)] $-G\widehat{F}_{2}\in \Sigma^{+}(H)$, 
    \item[(iii)] $\Lambda(t)=D^{*}\Pi(t) F_{2}+\Pi(t) F_{1}-M\widehat{F}_{1} \in \Sigma^{+}(H)$ for $\lambda$-a.e.\ $t \in \mathcal{T}$.
\end{itemize}
\end{assumption}
\begin{theorem}
Suppose that \Cref{det-diff}--\Cref{ass63} hold. The FBSEEs given by \eqref{Riccati-AT}--\eqref{Mean-Field-AT} admit a unique solution.
\end{theorem}
\begin{proof} Let 
$(\bar{x}_1,(q_1,\widetilde{q}_1))$ and $(\bar{x}_2,(q_2,\widetilde{q}_2))$ 
be two distinct mild solutions of the FBSEE \eqref{Offset-AT}--\eqref{Mean-Field-AT}.
Define
$$
\bar{x}^{\lozenge}:=\bar{x}_1-\bar{x}_2,\qquad
q^{\lozenge}:=q_1-q_2,\qquad
\widetilde{q}^{\lozenge}:=\widetilde{q}_1-\widetilde{q}_2.
$$
By taking the difference between the corresponding components of the FBSEE \eqref{Offset-AT}--\eqref{Mean-Field-AT} satisfied by the solutions $(\bar{x}_1,(q_1,\widetilde{q}_1))$ and $(\bar{x}_2,(q_2,\widetilde{q}_2))$, we obtain, for all $t\in\mathcal{T}$,
\begin{align}
q^{\lozenge}(t)
&= S^{*}(T-t)G\widehat{F}_{2}\,\bar{x}^{\lozenge}(T)
   -\int_{t}^{T}S^{*}(r-t)\Big(\Pi(r)BB^{\ast}q^{\lozenge}(r)+\Lambda(r)\bar{x}^{\lozenge}(r)\Big)\,dr
   \notag \\ & \hspace{8.5cm}-\int_{t}^{T}S^{*}(r-t)\,\widetilde{q}^{\lozenge}(r)\,dW_0(r), \label{diff-q}\\
\bar{x}^{\lozenge}(t)
&= -\int_{0}^{t}S(t-r)\Big((BB^{\ast}\Pi(r)-F_1)\bar{x}^{\lozenge}(r)-BB^{\ast}q^{\lozenge}(r)\Big)\,dr. \label{diff-x}
\end{align}
Moreover, the Yosida approximating sequence $\big(\bar{x}^{\lozenge}_n,(q^{\lozenge}_n,\widetilde{q}^{\lozenge}_n)\big)$ of 
$\big(\bar{x}^{\lozenge},(q^{\lozenge},\widetilde{q}^{\lozenge})\big)$, for each $n\in\mathbb{N}$ and $t\in \mathcal{T}$, satisfies the FBSEEs
\begin{align}
q^{\lozenge}_n(t)
&= S_n^{*}(T-t)\,G\widehat{F}_{2}\,\bar{x}^{\lozenge}_n(T)
   -\int_{t}^{T} S_n^{*}(r-t)\Big(\Pi(r)BB^{\ast}q^{\lozenge}_n(r)+\Lambda(r)\bar{x}^{\lozenge}_{n}(r)\Big)\,dr
 \notag  \\ & \hspace{8.5cm} -\int_{t}^{T} S_n^{*}(r-t)\,\widetilde{q}^{\lozenge}_n(r)\,dW_0(r), \label{diff-qn}\\
\bar{x}^{\lozenge}_n(t)
&= -\int_{0}^{t} S_n(t-r)\Big(\big(BB^{\ast}\Pi(r)-F_1\big)\bar{x}^{\lozenge}_n(r)-BB^{\ast}q^{\lozenge}_n(r)\Big)\,dr, \label{diff-xn}
\end{align}
which admits a unique strong (hence mild) solution. Applying Itô's formula to $\langle q^{\lozenge}_n(t), \bar{x}^{\lozenge}_n(t) \rangle$ yields
\begin{equation}
d\langle q^{\lozenge}_n(t),\bar{x}^{\lozenge}_n(t)\rangle
= \Big[
\langle \Lambda \bar{x}^{\lozenge}_n(t),\bar{x}^{\lozenge}_n(t)\rangle
+\langle BB^{\ast}q^{\lozenge}_n(t), q^{\lozenge}_n(t)\rangle + \langle F_1 \bar{x}^{\lozenge}_n(t), q_n(t)\rangle
\Big]dt
- \langle \widetilde{q}^{\lozenge}_n(t),\bar{x}^{\lozenge}_n(t)\rangle\,dW_0(t),
\end{equation}
where $\langle q^{\lozenge}_n(0), \bar{x}^{\lozenge}_n(0) \rangle=0$ and $\langle q^{\lozenge}_n(T), \bar{x}^{\lozenge}_n(T) \rangle=\,\langle G\widehat{F}_{2}\,\bar{x}^{\lozenge}_{n}(T),\bar{x}^{\lozenge}_{n}(T)\rangle$. Integrating from $0$ to $T$ and taking expectations, we obtain 
\begin{align}
\mathbb{E}\,\langle q^{\lozenge}_n(T),\bar{x}^{\lozenge}_n(T)\rangle
= \mathbb{E}\!\int_0^T\!\Big[
\langle \Lambda(r)\,\bar{x}^{\lozenge}_n(r),\bar{x}^{\lozenge}_n(r)\rangle
+\langle BB^{\ast}q^{\lozenge}_n(r), q^{\lozenge}_n(r)\rangle + \langle F_1 \bar{x}^{\lozenge}_n(r), q^{\lozenge}_n(r)\rangle
\Big]\,dr .
\end{align}
Moreover, the Yosida approximating sequence $\big(\bar{x}^{\lozenge}_n,(q^{\lozenge}_n,\widetilde{q}^{\lozenge}_n)\big)$ of 
$\big(\bar{x}^{\lozenge},(q^{\lozenge},\widetilde{q}^{\lozenge})\big)$, for each $n\in\mathbb{N}$ and $t\in \mathcal{T}$, satisfies
\begin{align}
q^{\lozenge}_n(t)
&= S_n^{*}(T-t)\,G\widehat{F}_{2}\,\bar{x}^{\lozenge}(T)
   -\int_{t}^{T} S_n^{*}(r-t)\Big(\Pi(r)BB^{\ast}q^{\lozenge}_n(r)+\Lambda(r)\bar{x}^{\lozenge}(r)\Big)\,dr
 \notag  \\ & \hspace{8.5cm} -\int_{t}^{T} S_n^{*}(r-t)\,\widetilde{q}^{\lozenge}_n(r)\,dW_0(r), \label{diff-qn}\\
\bar{x}^{\lozenge}_n(t)
&= -\int_{0}^{t} S_n(t-r)\Big(\big(BB^{\ast}\Pi(r)-F_1\big)\bar{x}^{\lozenge}_n(r)-BB^{\ast}q^{\lozenge}(r)\Big)\,dr, \label{diff-xn}
\end{align}which admits a unique strong (hence mild) solution. Applying Itô's formula to $\langle q^{\lozenge}_n(t), \bar{x}^{\lozenge}_n(t) \rangle$ yields
\begin{equation}
d\langle q^{\lozenge}_n(t),\bar{x}^{\lozenge}_n(t)\rangle
= \Big[
\langle \Lambda \bar{x}^{\lozenge}(t),\bar{x}^{\lozenge}_n(t)\rangle
+\langle BB^{\ast}q^{\lozenge}(t), q^{\lozenge}_n(t)\rangle + \langle F_1 \bar{x}^{\lozenge}_n(t), q^{\lozenge}_n(t)\rangle
\Big]dt
- \langle \bar{x}^{\lozenge}_n(t),\widetilde{q}^{\lozenge}_n(t)dW_0(t)\rangle\,,
\end{equation}
where $\langle q^{\lozenge}_n(0), \bar{x}^{\lozenge}_n(0) \rangle=0$ and $\langle q^{\lozenge}_n(T), \bar{x}^{\lozenge}_n(T) \rangle=\,\langle G\widehat{F}_{2}\,\bar{x}^{\lozenge}(T),\bar{x}^{\lozenge}_{n}(T)\rangle$. Integrating from $0$ to $T$ and taking expectations, we obtain 
\begin{align}
\mathbb{E}\,\langle q^{\lozenge}_n(T),\bar{x}^{\lozenge}(T)\rangle
= \mathbb{E}\!\int_0^T\!\Big[
\langle \Lambda(r)\,\bar{x}^{\lozenge}(r),\bar{x}^{\lozenge}_n(r)\rangle
+\langle BB^{\ast}q^{\lozenge}(r), q^{\lozenge}_n(r)\rangle + \langle F_1 \bar{x}^{\lozenge}_n(r), q^{\lozenge}_n(r)\rangle
\Big]\,dr .
\end{align}
Taking the limit as $n\to\infty$, we get
\begin{align} \label{unif}
\mathbb{E}\,\langle G\widehat{F}_{2}\,\bar{x}^{\lozenge}(T),\bar{x}^{\lozenge}(T)\rangle
= \mathbb{E}\!\int_0^T\!\Big[
\langle \Lambda(r)\,\bar{x}^{\lozenge}(r),\bar{x}^{\lozenge}(r)\rangle
+\langle BB^{\ast}q^{\lozenge}(r), q^{\lozenge}(r)\rangle + \langle F_1 \bar{x}^{\lozenge}(r), q^{\lozenge}(r)\rangle
\Big]\,dr.
\end{align}
Note that \Cref{ass63} implies that both sides of \eqref{unif} are zero. It follows that 
\begin{equation} \label{iii2}
BB^{\ast}q^{\lozenge}(t)=0, \quad \mathbb{P}\otimes\lambda\text{-a.e.}
\end{equation}
The final conclusion follows by substituting \eqref{iii2} into \eqref{diff-x}, which yields
\[
\bar{x}^{\lozenge}(t)=0,\qquad \mathbb{P}\text{-a.s. for each }t\in\mathcal{T},
\]
and subsequently obtaining $(q^{\lozenge}(t)=0,\;\widetilde{q}^{\lozenge}(t)=0)$, $\mathbb{P}\otimes\lambda$-a.e.\ from \eqref{diff-q}.
\end{proof}

\subsection{Discussion of Assumptions}
While \Cref{det-diff} and \Cref{ass63}-(i) are clear, we provide sufficient conditions under which \Cref{as61}, \Cref{as62}, \Cref{ass63}-(ii), and \Cref{ass63}-(iii) hold.

Clearly, \Cref{ass63}-(ii) holds if $\widehat{F}_2 \le 0$ and $\widehat{F}_2$ commutes with $G$. Moreover, given that  \Cref{ass63}-(i) holds, condition \Cref{ass63}-(iii) holds if, for instance, $D = F_2$, 
which implies that $D^{*}\Pi(t)F_2 \in \Sigma^{+}(H)$, and in addition $M\widehat{F}_1 \in \Sigma(H)$ and 
$D^{*}\Pi(t)F_2 - M\widehat{F}_1 \geq 0$ for $\lambda$-a.e.\ $t \in \mathcal{T}$. 
The most straightforward example arises when \( D = F_2 = 0 \) and \( M\widehat{F}_1 \leq 0 \), as indicated in \cite{federico2024linear} for an MFG model without common noise and with a deterministic diffusion coefficient, a setting that is also relevant for the framework considered in this paper. Therefore, putting everything together, \Cref{ass63}‑(ii) and \Cref{ass63}‑(iii) hold if we assume that $D = F_2 = 0$, $\widehat{F}_1 \le 0$ and $\widehat{F}_1$ commutes with $M$, and that $\widehat{F}_2 \le 0$ and $\widehat{F}_2$ commutes with $G$. 
We note that the conditions $\widehat{F}_1 \le 0$ and $\widehat{F}_2 \le 0$ correspond to the case in which a representative agent aims to keep its state away from the average state of the population, which is particularly relevant in congestion or crowd-avoidance contexts.
  
Moreover, \Cref{as61} and \Cref{as62} hold whenever \Cref{ass63} is satisfied. Specifically, $\eta_n$ and $\eta$ are the unique strict and mild solutions of the operator-valued Riccati equations given by
\begin{align}
&\dot\eta_n(t)
\;+\;\eta_n(t)\big(A_n - BB^{*}\Pi(t)\big)
\;+\;\big(A_n^{*} - \Pi(t)BB^{*}\big)\eta_n(t) \nonumber\\
&\hspace{6.5cm}
\;-\;\eta_n(t)BB^{*}\eta_n(t)
\;+\;\Lambda(t)
\;=\;0,
\qquad \eta_n(T)=-G\,\widehat F_2, \label{etan1}\\
&\dot\eta(t)
\;+\;\eta(t)\big(A - BB^{*}\Pi(t)\big)
\;+\;\big(A^{*} - \Pi(t)BB^{*}\big)\eta(t) \nonumber\\
&
\hspace{6.5cm}-\;\eta(t)BB^{*}\eta(t)
\;+\;\Lambda(t)
\;=\;0,
\qquad \qquad \eta(T)=-G\,\widehat F_2, \label{eta}
\end{align}
respectively. When \Cref{ass63} holds, \eqref{etan1} and \eqref{eta} fall into the category of positive-type operator-valued Riccati equations \cite{ichikawa1979dynamic, guatteri2005backward,tessitore1992some, federico2024linear}. We also note that these equations can be written as the difference between two classical operator-valued Riccati equations that do not involve the term \( BB^{\ast}\Pi(t) \). For instance, the solution $\eta(t)$ to \eqref{eta} can be expressed as the difference $ R(t) - \Pi(t)$, 
where $\Pi(t)$ satisfies \eqref{Riccati-AT}
and $R(t)$ satisfies 
\begin{align}
\dot R(t) + R(t)A + A^{*}R(t) - R(t)BB^{*}R(t) + \Lambda(t) + D^{\ast}\Pi(t)D + M &= 0,
& R(T) &= G - G\,\widehat F_2.
\end{align}
This decomposition allows one to derive \(\eta_n\) as the Yosida approximation sequence of \(\eta\) 
and to obtain the corresponding convergence result indicated in \Cref{as62}, 
following, for instance, the arguments in \cite[Prop.~2.2]{tessitore1992some}.

\begin{remark}(Well-Posedness of Hilbert Space-Valued LQ MFGs without Common Noise over Arbitrary Finite Time Horizons)
In the absence of common noise, i.e. when \( \sigma_0 = 0 \) in addition to \Cref{det-diff}, it can be shown that under \Cref{as61} and \Cref{as62}, there exists a solution to the corresponding mean-field consistency condition, expressed as a set of coupled forward-backward deterministic evolution equations, over an arbitrary finite time horizon. Moreover, if \Cref{ass63} holds, this solution is unique. Hence, the analysis developed in this section extends the well-posedness results of \cite{liu2025hilbert}, which were established only for small time horizons.
\end{remark}  
\begin{remark}(Nash and 
$\epsilon$-Nash Equilibria over Arbitrary Finite Time Horizons) Under \Cref{init-cond-iid}--\Cref{Info-set-Adm-Cntrl-CN} and \Cref{det-diff}--\Cref{ass63}, the results of \Cref{Nash-eqcom} and \Cref{E-Nash-thm} hold, in particular without the need for the contraction condition \eqref{contrcom}.
\end{remark}

\section{Concluding Remarks}
We developed the theory of LQ MFGs 
with common noise in Hilbert spaces. 
The presence of common noise introduces new analytical challenges, 
as the mean-field consistency condition naturally leads 
to a fully coupled system of forward-backward stochastic evolution equations in Hilbert spaces. 
We established the well-posedness of such systems over both small and arbitrary finite time horizons.

\bibliographystyle{elsarticle-num-names} 

\bibliography{sample}

\end{document}